\documentclass[11pt, oneside, labels]{amsart}   	
\usepackage{geometry}                		
\usepackage{amsmath} 
\usepackage{amsthm} 
\usepackage{amssymb}	
\usepackage{graphicx} 

\usepackage{float}

\usepackage{color}

\usepackage{enumitem}

\newtheorem{theorem}{Theorem}[section]
\newtheorem{lemma}[theorem]{Lemma}
\newtheorem{corollary}[theorem]{Corollary}
\newtheorem{prop}[theorem]{Proposition}
\theoremstyle{definition}
\newtheorem{dfn}{Definition}

\theoremstyle{remark}
\newtheorem*{rmk}{Remark}

\newcommand{\R}{{\mathbb R}}
\newcommand{\B}{{\mathbb B}}
\newcommand{\sph}{{\mathbb S}}
\renewcommand{\div}{\text{\rm div}\, }
\newcommand{\grad}{\, \text{\rm grad}\,}
\newcommand{\bs}[1]{{\boldsymbol  #1}}
\newcommand{\bb}[1]{{\mathbf  #1}}
\newtheorem{exa}{Example}[section]

\DeclareMathOperator{\rnge}{range}

\title[Inverse Potential Problems  
with Total Variation Regularization]{Inverse Potential Problems for 
Divergence
of Measures with Total Variation Regularization}

    \author{L. Baratchart}
   \address{Projet APICS, INRIA, 2004 route des Lucioles, BP 93,
Sophia-Antipolis, 06902 Cedex, FRANCE}
\email{Laurent.Baratchart@inria.fr}
  \author{C. Villalobos Guillen}
\address{Department of Mathematics,
Vanderbilt University,
Nashville, TN 37240, USA}
\email{ cristobal.villalobos.guillen@Vanderbilt.Edu}

  \author{D.P.~Hardin}
\address{Department of Mathematics,
Vanderbilt University,
Nashville, TN 37240, USA}
\email{doug.hardin@vanderbilt.edu}

\author{M.C. Northington}
\address{School of Mathematics, 
Georgia Institute of Technology, 
Atlanta, GA 30313
}
\email{mcnv3@gatech.edu}
  \author{E.B.~Saff}
\address{Department of Mathematics,
Vanderbilt University,
Nashville, TN 37240, USA}
\email{edward.b.saff@vanderbilt.edu}

\thanks{This research  was supported, in part, by the U.\ S.\ National Science Foundation under
grant  DMS-1521749, and by the INRIA grant to the associate team IMPINGE.}

  
\setenumerate[0]{label={\rm (\alph*)}}

\begin{document}

\begin{abstract}
	We study  inverse problems for the Poisson equation with source term the divergence of  an $\R^3$-valued measure; that is, the potential $\Phi$ satisfies
	$$
	\Delta \Phi= \div \bs{\mu},
	$$
	and $\bs{\mu}$ is to be reconstructed 
knowing (a component of) the field $\grad \Phi$ on a set disjoint from the support of $\bs{\mu}$.  Such problems arise in several electro-magnetic contexts
in the quasi-static regime, for instance
when recovering a remanent magnetization from measurements of its magnetic field.   We develop methods for  recovering $\bs{\mu}$  based on total variation regularization.   We provide sufficient conditions for the unique recovery of
$\bs{\mu}$, asymptotically when the regularization parameter and the noise 
tend to zero in a combined fashion,  when it is  uni-directional or  when the magnetization has a support 
	which is sparse in the sense that it is purely 1-unrectifiable.   
	
	Numerical examples are provided to illustrate the main theoretical results.   
\end{abstract}

\maketitle

\section{Introduction}
\label{intro}
This work is concerned with inverse potential problems with source term 
in divergence form. That is,
a $\R^3$-valued measure on $\R^3$ has to be recovered knowing
(one component of) the field of the Newton potential of its divergence on 
a piece of surface, away from the support. 
Such issues typically arise in source identification  
from field measurements for Maxwell's equations, in the quasi-static 
regime. They occur for instance
in electro-encephalography (EEG), 
magneto-encephalography (MEG), geomagnetism 
and paleomagnetism, as well as in several non-destructive
testing problems, see {\it e.g.} \cite{BaMoLe,BaKo,Blakely,Parker,KreKuPo}
and their bibliographies. A model problem of our particular interest is
inverse scanning magnetic microscopy, as considered for instance in
\cite{IMP,IMP2,BCL-JE} to recover magnetization distributions of 
thin rock samples, but the considerations below are of a more general 
and abstract 
nature. Our main objective is to introduce 
notions of sparsity that  help recovery in this infinite-dimensional context,
when regularization is performed by penalizing the total variation of the 
measure.


\subsection{Two Basic Extremal Problems}

For a closed subset   $S\subset \R^3$, let $\mathcal{M}(S)$ denote the space of
finite signed Borel measures on $\R^3$ whose support lies in $S$. 
In this paper we consider inverse problems associated with the equation
\begin{equation}\label{equationPhi}
\Delta \Phi=\div \bs{\mu}  
\end{equation} 
where $\bs{\mu}$ is an unknown measure in $\mathcal{M}(S)^3$ to be recovered.  Under suitable conditions on the decay of $\Phi$ at infinity, there is a unique solution $\Phi=\Phi(\bs{\mu})$  (see Section~\ref{Sect3D}).   Such problems arise in magnetostatics (our primary motivation) where  
$\bs{\mu}$ models a   {\em magnetization} distribution.  Then   
\begin{equation}\label{bextension}
\bb{b}(\bs{\mu})=\mu_0\left(\bs{\mu}-\grad\Phi(\bs{\mu})\right)
\end{equation} 
is the   {\em magnetic   field} $\bb{b}(\bs{\mu})$ generated by   $\bs{\mu}$ and it follows from \eqref{equationPhi} that $\bb{b}(\bs{\mu})$ is divergence-free.   
The term $\bb{h}(\bs{\mu}):=\mu_0\grad\Phi(\bs{\mu})$ is called the {\em magnetic intensity} generated by $\bs{\mu}$.  
 We refer to \eqref{equationPhi} as a
{\em Poisson-Hodge} equation since \eqref{bextension} provides a decomposition of $\bs{\mu}$ into curl-free and divergence-free terms.

The mapping $\bs{\mu}\to \bb{b}(\bs{\mu})$ is, in general, not injective.  We say that magnetizations $\bs{\mu},\bs{\nu}\in\mathcal{M}(S)^3$ are {\em $S$-equivalent} if  $\bb{b}(\bs{\mu})$ and $\bb{b}(\bs{\nu})$ agree on $\R^3\setminus S$, in which case we write
$$
\bs{\mu}\stackrel{S}{\equiv}  {\bs{\nu}}.
$$
A magnetization $\bs{\mu}$ is said to be {\em $S$-silent} (or {\em silent in $\R^3\setminus S$}) if it is $S$-equivalent to the zero magnetization; i.e., if $\bb{b}(\bs{\mu})$ vanishes on $\R^3\setminus S$.   It is suggestive from \eqref{equationPhi} (see Theorem~\ref{divM} below) that
a divergence-free magnetization $\bs{\mu}\in\mathcal{M}(S)^3$ is $S$-silent.  For the converse direction we introduce the following terminology.
\begin{dfn}
\label{defslender}
	We will call a closed set $S\subset\R^3$   \emph{slender} (with respect to $\R^3$) if its Lebesgue measure $\mathcal{L}_3(S)=0$ and 
	each   connected component $C$ of $\R^3\setminus S$ has   $\mathcal{L}_3(C)=\infty$.
\end{dfn}   
For example, a subset of a plane is slender in $\R^3$, while a sphere or a 
ball is not.  

Slender sets form a family of $S$ for which the above-mentioned
inverse problem makes contact 
with geometric measure theory, since Theorem~\ref{divM} shows that if $S$ is a slender set, then any $S$-silent 
magnetization is divergence-free. 
Smirnov \cite{Smi94}  characterizes divergence-free magnetizations (also known as {\em solenoids}) in $\R^3$ in terms of magnetizations that are absolutely continuous with respect to 1-dimensional Hausdorff measure, and we shall make
use of this characterization.

We will assume that scalar data  of the form   
$f=A(\bs{\mu}):= v\cdot\bb{b}(\bs{\mu})$, for some  fixed nonzero vector $v\in\R^3$, is given on a closed subset $Q\subset \R^3\setminus  S$, where $A$ is
the \emph{forward operator} mapping $\bs{\mu}$ to the restriction on $Q$
of $v\cdot\bb{b}(\bs{\mu})$. We will consider the situation where
\begin{itemize}
	\item[(a)] $A: \mathcal{M}(S)^3\to L^2(Q)$ boundedly and
	\item[(b)] $A(\bs{\mu})=0$ if and only if $\bs{\mu}$ is $S$-silent.
\end{itemize}   Since $\bb{b}(\bs{\mu})$ is harmonic on   $\R^3\setminus S$, condition (a) will hold if  $Q$ and $S$ are positively separated, 
and if $Q$ is compact (which is the case in practice), then 
$A$ is a compact operator.  
Theorem~\ref{Lem4} provides sufficient conditions for  (b) to hold. 
Condition
(b) means that the observation 
is ``faithful'', {\it i.e.} if the $v$-component of the field on $Q$ is zero then
the field is indeed zero everywhere off $S$. In this case, the null space of
$A$ (which is a crucial ingredient of the inverse problem) coincides 
with $S$-silent magnetizations which depend solely on the geometry of $S$ and
can be studied using potential and measure-theoretic tools. 
As we shall see,  $A$ has, in general, a nontrivial null 
space  and, if $Q$ is ``thin enough'',
it has dense range (see Lemma \ref{Acomp}).
Note  that a typical magnetic sensor is a coil measuring the component of
the field parallel to its axis, which is why measurements are assumed to be of the form $v\cdot \bb{b}(\bs{\mu})$. The fact that $v$ is a constant vector 
means that the orientation of the sensor is kept fixed. In some cases,
{\it e.g.} in magneto-encephalography, $v$ would   depend
on the point where a measurement  is made. We do not consider this 
more complicated situation, as we are particularly motivated by 
applications to 
scanning magnetic microscopy (SMM) 
where measurements of the vertical component of  the magnetic field, 
namely $b_3(\bs{\mu})=e_3\cdot \bb{b}(\bs{\mu})$ ($e_i$  denotes
the $i$-th unit vector 
of the canonical basis for $\R^3$, for $i=1,2,3$),   are taken on a rectangle $Q$ in a plane $x_3=h$ for some $h>0$, while $S$ is contained in the half-space $x_3\le 0$.  

The focus of this paper is on the use of the total variation norm for measures to regularize the ill-posed inverse problem of recovering $\bs{\mu}$ from $A(\bs{\mu})$.  Let us briefly recall that the total variation measure of $\bs{\mu}\in  \mathcal{M}(S)^3$  is the positive Borel measure   $|\bs{\mu}|\in   \mathcal{M}(S)$ such that 
$d\bs{\mu}=\bb{u}_{\bs{\mu}}d|\bs{\mu}|$  for some $\R^3$-valued Borel measurable function $\bb{u}_{\bs{\mu}}$ satisfying $|\bb{u}_{\bs{\mu}}|=1$ at  $|\bs{\mu}|$-a.e. point, see \eqref{totVarMeasDef}. The {\em total variation norm of $\bs{\mu}$} is then defined as
\begin{equation}\label{TVnormdef}\|\bs{\mu}\|_{TV}:=|\bs{\mu}|(\R^3).\end{equation}
If $\bb{u}_{\bs{\mu}}$ is constant  $|\bs{\mu}|$-a.e., then we   call $\bs{\mu}$ {\em uni-directional}.  Such magnetizations arise in geological samples whose remanent magnetizations were formed in a uniform external field.


We investigate two extremal problems for magnetization recovery.  The first  is that of minimizing the total variation  over  magnetizations  $S$-equivalent to a given one.  To fix notation, for $\bs{\mu}\in \mathcal{M}(S)^3$, let
\begin{equation}\label{defmineq}
M(\bs{\mu}):=\inf \{ \|\bs{\nu}\|_{TV}\colon \bs{\nu} \stackrel{S}{\equiv}  {\bs{\mu}}\}.
\end{equation}

\noindent {\bf Extremal Problem 1 (EP-1).}
For $\bs{\mu}_0\in \mathcal{M}(S)^3$, find $\bs{\mu}\stackrel{S}{\equiv} \bs{\mu}_0$   such that  $\|\bs{\mu}\|_{TV}=M(\bs{\mu}_0)$.
\smallskip

The second extremal problem involves minimizing the following functional
defined for $\bs{\mu}\in\mathcal{M}(S)^3$, $f\in L^2(Q)$, and $\lambda>0$, by
\begin{equation}
\label{defcrit0}
\mathcal{F}_{f,\lambda}(\bs{\mu}):=
\|f-A\bs{\mu}\|_{L^2(Q)}^2+\lambda \|\bs{\mu}\|_{TV}.
\end{equation}
We consider the problem of finding some  $\bs{\mu_\lambda}\in \mathcal{M}(S)^3$  minimizing $\mathcal{F}_{f,\lambda}$: \\

\noindent {\bf Extremal Problem 2 (EP-2).}
Given   $f\in L^2(Q)$, find $\bs{\mu}_\lambda\in\mathcal{M}(S)^3$ such that
\begin{equation}
\label{crit0}
\mathcal{F}_{f,\lambda}(\bs{\mu_\lambda})=\inf_{\bs{\mu}\in\mathcal{M}(S)^3} \mathcal{F}_{f,\lambda}(\bs{\mu}).
\end{equation}
\smallskip

We remark that the total variation norm is convex on $\mathcal{M}(S)^3$ but not strictly convex and so there may be multiple   $\bs{\mu}$ that solve  EP-1 for a given $\bs{\mu}_0$.  
Still, we show in Section~\ref{Sect3D} that, under the assumption that $S$ is slender,  EP-1 uniquely recovers the magnetization in two  important cases: (a)  the magnetization is supported on
a purely 1-unrectifiable set  (see  Theorem \ref{unrec3D}) 
and (b) the magnetization is piecewise uni-directional (see Theorem \ref{UnidirThm}
).  The notion of purely 1-unrectifiable set is classical 
from geometric measure 
theory, and  means that intersection of the set with any rectifiable arc has 1-dimensional Hausdorff measure  zero.  In particular, any set with 1-dimensional Hausdorff measure zero  is purely 1-unrectifiable and so   EP-1   recovers a   large class of magnetizations  with `sparse' support including, for example, those with  countable or   Cantor-like (with Hausdorff dimension less than 1) support.   Note that uni-directionality involves some sort of sparsity as well,
this time direction-wise.

In Section~\ref{Sect3D}, we also consider the {\em net moment} of $\bs{\mu}$ given by
\begin{equation}\label{netmomdef}\langle \bs{\mu}\rangle:=\bs{\mu}(\R^3),\end{equation}
which is   a quantity of physical interest and is useful in our analysis of uni-directional magnetizations.
Under the assumption that $S$ is compact or a  slender set,
Theorem~\ref{divM} and  Lemma~\ref{MomZero}
show that $S$-equivalent magnetizations must have the same net moment.
Hence, $\langle \bs{\mu}\rangle$ is uniquely defined by 
the field in this case. However,
Remark \ref{momsil} after  that proposition shows this needs not hold
if $S$ is neither compact nor slender.

Solutions to  EP-2 connect to those of EP-1 as follows.
If
$f=A(\bs{\mu})+e$ with `noise' $e\in L^2(Q)$, any weak-* accumulation point 
of the solutions to EP-2 when $\lambda\to0$ and $e/\sqrt{\lambda}\to0$ must 
be a solution of EP-1. This Tikhonov-like regularization theory is by now
essentially understood in a more general context
\cite{BurOsh,HoKaPoSch,BrePikk}. 
In Section \ref{regu},
we improve on previous results by showing
that this property holds not only for the $\R^3$-valued measures involved but 
also for their total variation measures, and that weak-* convergence 
strengthens to narrow convergence (see Theorem \ref{lamto0Thm}). 
Hence, ``no mass is lost'' in the limit. This is important for  if
EP-1 has a unique solution, it implies that solutions to EP-2 
asymptotically put mass exactly where the ``true''  magnetization does, 
a property which is usually not satisfied under weak-$*$ convergence
of measures.
Another feature of solutions to EP-2, 
which is more specific to the present situation, is that
they  are supported on ``small'' sets, of codimension at least 1 in $S$
(see Corollary \ref{corspar}).

Altogether, when the ``true'' magnetization  is sparse in one of
the senses mentioned above (that is, if it can be recovered by EP-1), 
we obtain  in Theorem \ref{weak*to0} and Corollary \ref{points} asymptotic
recovery results which, from the strict point of view of inverse problems,
recap the main contributions of the paper.
Finally, let us mention that EP-2 has a unique solution when $S$ is 
contained in a plane, even though $\mathcal{F}_{f,\lambda}$ is not strictly 
convex  if  $A$ has a nontrivial null space
(the usual case considered in this paper).
A proof would take us too far into the structure
of divergence-free measures in the plane, and is beyond the scope of 
this paper, see \cite{BHV}.
Whether uniqueness holds more generally is an intriguing open question.


\subsection{Background and Related Work}

The class of  inverse problems investigated below can be 
viewed as  deconvolution  from partial observations.  
Note that we work with vector-valued 
measures supported on sets of dimension greater than 1, and 
the forward operator $A$,
whose convolution kernel is (weakly) singular, has a rather large null-space
which depends in a fairly complicated manner on the geometry of the observation set $Q$  and of the set  $S$
{\it a priori} 
containing the support of the unknown measure. 



The main contribution of the present work is perhaps to connect geometric measure 
theory with regularization theory for such inverse problems.
This connection essentially rests on the structure
of the null-space of the 
forward operator: in fact, the conditions on $Q$ and $S$ 
set up in  this work for unique recovery of sparse measures to hold
are designed so that
this  null-space
consists exactly of divergence-free measures. 
These conditions generalize those given in \cite{IMP} for planar distributions to magnetizations supported on more general sets in $\R^3$.  
This characterization of the null-space is central to the present approach,
and allows us to rely  on classical tools from geometric measure theory
and on material from \cite{Smi94} 
to proceed with the proof of our main result about  EP-1, namely the 
minimality of the total variation of a sparse measure in its coset modulo 
the null-space.
More general situations, notably the case where $S$ is a closed surface
or where it has positive Lebesgue 
measure
in $\R^3$ (and thus is not slender),  are left 
for future research. Such situations are relevant in particular
to EEG and MEG  inverse problems, as well as to SMM
on volumic samples.

Whether the property of having purely 1-unrectifiable support qualifies a 
measure as being ``sparse'' is debatable: for instance the support could still
disconnect the space (like the Koch curve does in 2-D).
Nevertheless, 
it comprises standard notions of sparsity, such as being a 
finite sum of Dirac masses, which is why we consider 
purely 1-unrectifiability of the support as a generalized notion of sparsity in this context.
Moreover if $S$ is slender we shall see that another convenient
notion of sparsity is to be a finite
sum of measures with disjoint supports, each of which assumes constant
direction (see Theorem \ref{UnidirThm}). 
Thus, sparsity in the present context may also be envisaged ``direction-wise''.

Let us now briefly discuss connections to other works dealing with
sparse recovery.
After early studies \cite{DoLo,Tibshi} and the seminal work in
\cite{CRT,CRT1,Do,Do1}, approximately solving  underdetermined 
systems of linear equations in $\R^n$ by minimizing the 
residuals while penalizing the $l^1$-norm has proved to be quite successful  
in identification. In fact, under appropriate assumptions on the matrix of 
the system, this kind of approximation favors the recovery of sparse 
solutions, {\it i.e.} solutions having
a large number of zero components, when they exist. This  has resulted in
the theory of compressed sensing, which shows by and large
that  a sparse
signal can be recovered from much less linear observations 
than is {\it a priori} needed, see for example
\cite{FouRau} and the bibliography therein.

In recent years, natural analogs in infinite-dimensional settings
have been
investigated by several authors, but then the situation is much less understood.
A Tikhonov-like
regularization theory was developed in \cite{BurOsh,HoKaPoSch,BrePikk}
for linear equations whose unknown is a (possibly $\R^n$-valued) measure, by
minimizing the residuals while penalizing the total variation. 
As expected from the non-reflexive character of spaces of measures,
consistency estimates generally hold in a  rather weak sense, 
such as weak-$*$ convergence of subsequences to solutions of minimum total 
variation, or convergence  in the Bregman distance
when  the so-called source condition holds. Algorithms and proofs
typically rely on Fenchel duality, and reference
\cite{BrePikk} contains an extension of the soft thresholding algorithm
to the case where the unknown gets parametrized as 
a finite linear combination of Dirac masses,
whose location no longer lies on a fixed grid 
in contrast with the discrete case. References \cite{ClaKu,CaClaKu},
which deal with inverse source problems for elliptic operators,
dwell on the same circle of ideas but  suggest a different 
thresholding method, connected with
a Newton step, or else a finite element discretization of the equation 
having a linear combination of Dirac masses
amongst its solutions.
These methods yield contructive 
algorithms to approximate a solution of minimum total variation 
to the initial equation by a sequence of
discrete measures, which is always possible in theory since these are 
weak-$*$ dense in the space of measures supported on an open subset of
$\R^n$. To obtain an asymptotic recovery result,  in the weak-$*$ sense as the
regularizing parameter goes to zero, 
it remains to identify conditions on a measure
ensuring that it is the unique   element 
of least total variation
in its  coset modulo the null-space of the forward operator. 
The gist of compressed sensing is that, in finite-dimension, sparsity is 
such a condition for ``most'' operators.
In infinite-dimension, when total variation-penalization
replaces $l^1$-penalization,
it is unclear which assumptions
imply the desired uniqueness and why they should be connected with 
sparsity. Moreover, granted the diversity of operators involved in 
applications,  it is likely that such assumptions will
much depend on the situation under examination.

Still, a striking connection with sparsity, in the sense of 
being a sum of Dirac masses, was recently
established for 1-D deconvolution issues, 
where a train of spikes is to be recovered from filtered observation thereof
\cite{BhaRe,ClaKu,CaGa,CaFe,DuPe}. More precisely,
following the work in \cite{CaFe} which identifies a sum of 
Dirac masses as the solution of minimum total variation in the case of
an ideal low-pass filter, it was shown in \cite{DuPe} that 
a finite train of spikes can be recovered 
arbitrarily 
well by total variation regularization, when the noise and the regularization 
parameter are small enough, provided that the filter satisfies
nondegeneracy conditions and the spikes are sufficiently separated. 
Also, 
the algorithm in \cite{BrePikk} can be used for that purpose.
The present work may be viewed as investigating in higher dimension 
such deconvolution issues, the filter being now defined by the Biot-Savart law,
and Corollary \ref{points} may be compared to \cite[Thm. 2]{DuPe}.


\subsection{Notation}
We conclude this section with some details of our notation, as used above, for vectors and vector-valued functions, and with some preliminaries concerning measures and distributions.  For a vector $x$ in the Euclidean space $\R^3$,  we denote the $j$-th component of $x$ by $x_j$ and   the partial derivative with respect to $x_j$ by $\partial_{x_j}$.  For $x, y\in \R^3$, $x\cdot y$ and $|x |=\sqrt{x\cdot x}$ denote the usual Euclidean scalar product and norm, respectively.   We denote the   gradient with respect to $y$ by $\grad_y$  and drop the subscript 
 if the variable with respect to which we differentiate is unambiguous.
  By default, we consider vectors $x$ as column vectors; e.g., for $x\in \R^3$ we write 
$x=(x_1,x_2,x_3)^T$ where ``$T$'' denotes ``transpose''.
We  use bold symbols to represent vector-valued functions and measures, and   the corresponding nonbold symbols with subscripts to denote the respective components; e.g., $\bs{\mu}=(\mu_1, \mu_2, \mu_3)$ or $\bb{b}(\bs{\mu})=(b_1(\bs{\mu}),b_2(\bs{\mu}), b_3(\bs{\mu}))$. We let    $\delta_x$ stand  for the Dirac delta measure at a point $x\in \R^3$ and refer to a magnetization of the form $\bs{\mu}=v\delta_x$ for some $v\in \R^3$ as the {\em point dipole at $x$ with moment $v$}.

For $x\in\R^3$ and $R>0$, we let $\B(x,R)$ denote the open ball centered at $x$ with radius $R$ and $\sph(x,R)$ the boundary sphere.
Given a finite measure $\mu\in\mathcal{M}(\R^3)$ and a Borel set $E\subset\R^3$, we denote by $\mu\mathcal{b}E$ the finite measure obtained by restricting  $\mu$ to $E$ ({\it i.e.} for every Borel set $B\subset\R^3$, $\mu\mathcal{b}E(B):=\mu(E\cap B))$.

For $\bs{\mu}\in  \mathcal{M}(S)^3$, the {\em total variation measure} $|\bs{\mu}|$  is defined on Borel sets $B\subset \R^3$ by
\begin{equation}\label{totVarMeasDef}
|\bs{\mu}|(B):=\sup_{\mathcal{P}} \sum_{P\in \mathcal{P}} |\bs{\mu}(P)|, 
\end{equation}
where the supremum is taken over all finite Borel partitions $\mathcal{P}$ of $B$.   
Since $|\bs{\mu}|$ is a Radon measure, the Radon-Nikodym derivative $\bb{u}_{\bs{\mu}}:=d\bs{\mu}/d{|\bs{\mu}|}$ exists  and satisfies $|\bb{u}_{\bs{\mu}}|=1$ a.e.  with respect to $|\bs{\mu}|$.

We shall identify $\bs{\mu} \in \mathcal{M}(\R^k)^k$ with the linear form on 
$(C_c(\R^k))^k$ (the space of $\R^k$-valued continuous  functions  on $\R^k$  with compact support equipped with the sup-norm)
given by 
\begin{equation}
\label{defmesf}\langle \bb{f},  \bs{\mu}\rangle := \int \bb{f} \cdot d\bs{\mu}, \qquad \bb{f}\in (C_c(\R^k))^k.\end{equation}
The norm of the functional
\eqref{defmesf},  is $\|\bs{\mu}\|_{TV}$.
It extends naturally with the same norm to the space 
$(C_0(\R^k))^k$ of $\R^k$-valued continuous functions on $\R^k$ vanishing at infinity.

At places, we also identify $\bs{\mu}$ with the restriction of \eqref{defmesf}
to $(C^\infty_c(\R^3))^3$, where $C^\infty_c(\R^3)$ is
the space of $C^\infty$-smooth functions with compact support,
equipped with the usual topology 
making it an LF-space 
\cite{Schwartz}.   
We refer to a continuous 
 linear functional on $(C^\infty_c(\R^m))^n$ as being
a distribution.


We denote     Lebesgue measure on $\R^3$ by $\mathcal{L}_3$ and $d$-dimensional Hausdorff measure on $\R^3$ 
by $\mathcal{H}_d$ .  We normalize $\mathcal{H}_d$    for $d=1$ and $2$ so that it coincides with  arclength and   surface area for smooth curves and surfaces, respectively.  We denote    the Hausdorff dimension of a set $E\subset \R^3$ by $\dim_{\mathcal{H}}(E)$.

 \section{Equivalent  magnetizations, net moments, and total variation}\label{Sect3D}
 In this section, we discuss some  regularity issues for magnetic 
fields and potentials, and we study the connection between 
$S$-silent sources and $\R^3$-valued measures which are distributionally
divergence-free. This rests on the notion of a \emph{slender set} 
introduced in Definition~\ref{defslender}.
Subsequently, we solve Extremal Problem 1 
for certain classes of magnetizations when $S$ is slender. 
Such magnetizations are ``sparse'', in the sense that
either their support is purely 1-unrectifiable (see 
Theorem \ref{unrec3D}) or 
they assume a single direction on each piece of some finite partition 
of $S$ (see Theorem \ref{UnidirThm}). We also give conditions 
on $S$ and $Q$ ensuring that the forward operator has kernel the space of
$S$-silent magnetizations (see Lemma \ref{Lem4}).

We first define  the {\em magnetic field} $\bb{b}(\bs{\mu})$ and the   {\em scalar magnetic potential}  $\Phi(\bs{\mu})$ (see \cite{Jac1975}) generated by a magnetization distribution $\bs{\mu}$ at   points $x$ not in the support of  
$\bs{\mu}$   as 
\begin{equation}
\label{bDef1}
\begin{split}
\bb{b}(\bs{\mu})(x)&:=-c\left(\int \frac{1}{|x-y|^3}\, d\bs{\mu}(y) -3 \int (x-y)\frac{(x-y)\cdot d\bs{\mu}(y)}{|x-y|^5}\right) \\
&=-c\grad \int \grad_y\frac{1}{|x-y|}\cdot d\bs{\mu}(y) \\
&=-\mu_0\grad \Phi(\bs{\mu})(x), \qquad x\not\in \text{supp } \bs{\mu},
\end{split}
\end{equation}
where $c=10^{-7}Hm^{-1}$, $\mu_0 = 4 \pi \times c$, and $\Phi(\bs{\mu})$   is given by
\begin{equation}\label{PhiDef}
\Phi(\bs{\mu})(x):=\frac{1}{4\pi}\int \grad_y\frac{1}{|x-y|}\cdot d\bs{\mu}(y)=\frac{1}{4\pi}\int \frac{x-y}{|x-y|^3}\cdot d\bs{\mu}(y) , \qquad x\not\in \text{supp } \bs{\mu}.
\end{equation}     Note that the second equality in \eqref{bDef1} follows from  the smoothness of the 
Newton kernel $1/{|x-y|}$ for $x\neq y$ and  the formula
$$
\grad_x \left( \grad_y\frac{1}{|x-y|}\cdot a\right)=\frac{a}{|x-y|^3}  -3  (x-y)\frac{(x-y)\cdot a}{|x-y|^5},
$$ for a fixed $a\in\R^3$.
    Note that 
$\Phi(\bs{\mu})$ and the components of
$\bb{b}(\bs{\mu})$ are harmonic functions on $\R^3\setminus S$.  We show in Proposition~\ref{LocIntPhi} that $\Phi$ can   be extended to 
a locally integrable function on $\R^3$ that satisfies, in the sense of distributions, the Poisson-Hodge equation \eqref{equationPhi}.  Furthermore, this proposition shows that the magnetic field  $\bb{b}(\bs{\mu})$  extends 
to a divergence-free $\R^3$-valued distribution.

 \begin{prop}\label{LocIntPhi} 
Let $S$ be a closed proper subset of $\R^3$ and $\bs\mu\in\mathcal{M}(S)^3$.
Then, the integral in the right-hand side
of \eqref{PhiDef} converges absolutely
for a.e. $x\in\R^3$, thereby defining
an extension of $\Phi(\bs{\mu})$ to all of $\R^3$. 
Denoting this extension by $\Phi(\bs{\mu})$ again,
it holds 
for each $p$, $q$  with $1\le p<3/2<q\le \infty$ 
that $\Phi(\bs{\mu})\in L^p(\R^3)+L^q(\R^3)$,
and that, in the distributional sense, 
$$\Delta\Phi(\bs{\mu})=\div{\bs{\mu}}.$$ 

Furthermore, the distribution $\mu_0\left(\bs{\mu}-\grad\Phi(\bs{\mu})\right)$ is divergence-free and extends $\bb b(\bs{\mu})$ from a distribution
on $\R^3\setminus S$  to a distribution on $\R^3$. 
Denoting this extension by   $\bb b(\bs{\mu})$ again,
we have that
$\langle T_\alpha \bb f,  \bb b(\bs{\mu})\rangle \rightarrow 0$ as $|\alpha|\rightarrow\infty$ for every $\bb{f}\in (C_c^\infty(\R^3))^3$ (here $T_\alpha\bb f$ denotes  the translation of the argument of $\bb f$ by $\alpha$).
\end{prop}
\begin{rmk}  The decomposition $\bs{\mu}=\bb b(\bs{\mu})/\mu_0+\grad\Phi(\bs{\mu})$ is the Helmholtz-Hodge decomposition of the $\R^3$-valued measure
$\bs{\mu}$ into the sum of a 
gradient and a divergence-free term. Although $\bs{\mu}$ is a distribution 
of order 0, note that the summands will generally have order -1. 
Hereafter, we 
let $\Phi(\bs{\mu})$ and $\bb b(\bs{\mu})$ denote the extensions 
given in Proposition \ref{LocIntPhi}.
\end{rmk}

\begin{proof}
Set $\bb{k}(y)=y/(4\pi |y|^3)$ for $y\in\R^3$ and let $\phi\in C_c^\infty(\R^3)$ be valued in
	$[0,1]$, identically $1$ on $\B(0,1)$ and  $0$ outside of
	$\B(0,2)$. Writing $\bb{f}_1=\phi  \bb{k}$ and $\bb{f}_2=(1-\phi)\bb{k}$,
	we have that $|\bb{f}_1|\in L^p(\R^3)$ for $1\le p<3/2$ and $|\bb{f}_2|\in L^q(\R^3)$. From \eqref{PhiDef}, we  get that
	$\Phi(\bs{\mu})=\bb{k}*\bs{\mu}=\bb{f}_1*\bs{\mu}+\bb{f}_2*\bs{\mu}$ off $S$
	(the product under the convolution integral is here the scalar product).   
	For any $r\in[1,\infty]$, Jensen's inequality implies that
	the convolution of a finite signed measure with an $L^r$ function
	is an $L^r$ function with norm not exceeding the mass of the measure
	times the initial norm, and Fubini's theorem entails that the integrals converge absolutely a.e.   Therefore $\bb{f}_1*\bs{\mu}\in L^p(\R^3)$ and $\bb{f}_2*\bs{\mu}\in L^q(\R^3)$,
	showing that $\bb{k}*\bs{\mu}\in L^p(\R^3)+L^q(\R^3)$.	

	We next show   that
	\begin{equation}\label{DelPhi}
	\Delta\Phi(\bs{\mu})=\Delta(\bb{k}*\bs{\mu})=\div{\bs{\mu}}.
	\end{equation}   Let
	$N(y)=-1/(4\pi|y|)$ for $y\in\R^3$, 
	pick $\psi\in C_c^\infty(\R^3)$
	and recall that $G:=N*\psi$ is a smooth function vanishing at infinity
	such that $\Delta G=\psi$ \cite[Cor. 4.3.2\&4.5.4]{Gardiner}.
	Now, differentiating under the integral sign, we have that 
	$\bb{k}*\psi=\grad G$, therefore
	\[
	\begin{array}{ll}
	\langle\Delta\Phi(\bs{\mu}),\psi\rangle&=\langle \bb{k}*\bs{\mu},\Delta\psi\rangle
	=-\langle \bs{\mu},\bb{k}*\Delta\psi\rangle
	=-\langle \bs{\mu},\Delta(\bb{k}* \psi)\rangle
	=-\langle \bs{\mu},\Delta(\grad G)\rangle\\
	&=-\langle \bs{\mu},\grad\Delta G\rangle
	=\langle \div\bs{\mu},\Delta G\rangle
	=\langle \div\bs{\mu},\psi\rangle
	\end{array}
	\]
	which proves \eqref{DelPhi}.

	Finally, the distribution $\bb{d}:=\mu_0 (\bs{\mu}-\grad \Phi(\bs{\mu}))$ coincides with $\bb{b}(\bs{\mu})$ on $\R^3\setminus S$, by \eqref{bDef1}, and it follows from \eqref{DelPhi} that it satisfies
	$
	\div\bb{d}=\mu_0 (\div \bs{\mu}-\Delta \Phi(\bs{\mu}))=0$; i.e., $\bb{d}$ is divergence-free.  
	Renamimg $\bb{d}$ as $\bb{b}(\bs{\mu})$ since it extends the latter,
the finiteness of $\bs{\mu}$ and the fact that $\Phi(\bs\mu)\in L^1(\R^3)^3+L^2(\R^3)^3$  show:
	\begin{equation*}\langle T_\alpha \bb f,  \bb b(\bs{\mu})\rangle=\mu_0\left(\langle T_\alpha \bb f,  \bs{\mu}\rangle-\langle \div T_\alpha \bb f,   \Phi(\bs{\mu})\rangle\right)\to 0
	\end{equation*}
	as $|\alpha|\to \infty$, for every $\bb{f}\in (C_c^\infty(\R^3))^3$. 
\end{proof}

  \subsection{Divergence-free  and silent magnetizations}\label{secDivSilent}

The equation $\Delta\Phi(\bs{\mu})=\div{\bs{\mu}}$ is suggestive of the existence of a relationship between $S$-silent and divergence-free magnetizations.
As will be seen from the lemma below, for any closed set $S$ all divergence-free magnetizations supported on $S$ are $S$-silent, but the converse is not always true as follows from the next construction.

\begin{exa}\label{balldipole}
	Let $S=\mathbb{B}(0,1)$ be the closed unit Euclidean ball centered at the origin,
	$\mu=\mathcal{L}_3\mathcal{b}S$, and $\bs{\mu}\in\mathcal M(S)^3$   the $\R^3$-valued measure equal to $(4\pi/3)^{-1}\mu e_1$.	Then, by the mean value theorem, we get that
	$$
	\Phi(\bs{\mu})(x)=\frac{1}{4\pi}\int\grad_y\ \frac{1}{|x-y|}\cdot d\bs{\mu}(y)\\
	=\frac{1}{4\pi}\frac{x_1}{|x|^3},\qquad x\notin S,
	$$
	since $\frac{1}{4\pi}\frac{x_1}{|x|^3}$ is harmonic on $\R^3\setminus\{0\}$.
	Note that $\frac{1}{4\pi}\frac{x_1}{|x|^3}$ is also the magnetic potential generated by the dipole $\bs{\nu}:=\delta_0e_1$; therefore $\bs{\mu}$ and $\bs{\nu}$ are $S$-equivalent, that is $\bs{\mu}-\bs\nu$ is $S$-silent.
	However, this magnetization is not divergence-free since, for every $f\in   C_c^\infty(\R^3) $ supported in $\B(0,1)$, it holds that
$\langle  f, \div(\bs{\mu}-\bs{\nu})\rangle=-\langle f,\div\bs{\nu}\rangle=-\partial_{x_1} f(0)$.
	
	An analogous argument shows that, for $a_3$ the area of $\sph(0,1)$ and $\tilde\mu:=\mathcal{H}_2\mathcal{b}\sph(0,1)$, the $\R^3$-valued measure
$a_3^{-1}\tilde\mu e_1$ is likewise $S$-equivalent to $\bs\nu$.

The following variant of this example is also instructive: there is a
sequence $x_n\in \mathbb{B}(0,1)$ and a sequence $c_n$ of real numbers
with $\sum_n|c_n|<\infty$ such that the measure $\bs{\alpha}$ defined by
$\bs{\alpha}=e_1\sum_nc_n\delta_{x_n}$
is $S$-equivalent to $a_3^{-1}\tilde\mu e_1$ which is $S$-equivalent to $\bs\nu$.  Therefore 
$\bs{\alpha}-\bs{\nu}$ is an 
$S$-silent magnetization consisting of countably many point dipoles.
To see that $x_n$ and $c_n$ exist, recall Bonsall's theorem
(whose proof in the ball is the same as in the disk, see
\cite[Thms. 5.21 \& 5.22]{Rudin}) that 
whenever $x_n$ is a sequence in
$\mathbb{B}(0,1)$ which is nontangentially dense in $\sph(0,1)$,
each function $h\in L^1(\tilde\mu)$
can be written as $h(\xi)=\sum_n c_n P_{x_n}(\xi)$ 
where 
$P_{x_n}(\xi)=(1/4\pi)(1-|x_n|^2)/|\xi-x_n|^3$ is the familiar Poisson kernel 
of the unit ball at $x_n$, and $c_n$ is a sequence of real numbers with absolutely 
convergent sum. Choosing $h\equiv1$ and observing that, for $y\notin S$, 
\[
\frac{y-x_n}{|y-x_n|^3}=\int\frac{y-\xi}{|y-\xi|^3}P_{x_n}(\xi)\,
d\tilde\mu(\xi)
\]
by the Poisson representation of harmonic functions, we easily check that
$e_1\sum_nc_n\delta_{x_n}$
is $S$ equivalent to $a_3^{-1}\tilde\mu e_1$, as desired.
\end{exa}

Example \ref{balldipole} shows us that a $\R^3$-valued measure on $S$ which
is $S$-silent  needs not be divergence-free in general. It does, however, when $S$ is slender:

\begin{theorem}\label{divM}
Let $S\in\R^3$ be  closed and $\bs\mu\in\mathcal{M}(S)^3$.
If $\div\bs\mu=0$, then $\bs\mu$ is $S$-silent.
Furthermore, if $S$ is a slender set and  $\bs\mu$ is $S$-silent, then $\div\bs\mu=0$.
\end{theorem}
\begin{proof}
	Since $\Phi(\bs\mu)\in L^1(\R^3)^3+L^2(\R^3)^3$ by Proposition~\ref{LocIntPhi},
we get  from the Schwarz inequality:
\begin{equation}
\label{inegPhig}
\int_E|\Phi(\bs\mu)|d\mathcal{L}_3\leq C_1+C_2\,\mathcal{L}_3^{1/2}(E)
\end{equation}
for some constants $C_1$,  $C_2$ and each Borel set $E$ of finite measure. 
If $\div\bs\mu=0$, then   $\Phi(\bs\mu)$ is harmonic 
on $\R^3$ by the same lemma, therefore it is constant  by
\eqref{inegPhig} and the mean value theorem 
(see proof of \cite[Thm. 2.1]{ABR}). 
Consequently $\bs{\mu}$ is $S$-silent.

For the second statement,  assume that $S$ is a  slender set  and that $\bs\mu\in\mathcal{M}(S)^3$ is $S$-silent.
Since $\bb b(\bs\mu)=\grad \Phi(\bs\mu)$ and $ \bs\mu$ is $S$-silent, then $\Phi(\bs\mu)$ is constant on each connected component of $\R^3\setminus S$. 
If $\mathfrak{C}$ is such a component, we can apply \eqref{inegPhig} with
$E=\mathfrak{C}\cap B(0,n)$ and let $n\to\infty$ to conclude that the 
corresponding constant is zero,
because $\mathcal{L}_3(\mathfrak{C})=+\infty$. 
Hence, $\Phi(\bs{\mu})$ must be zero on $\R^3\setminus S$,
and since $\mathcal{L}_3(S)=0$ it follows that  $\Phi(\bs{\mu})$ is zero as a distribution, so that $\div\bs\mu=\Delta\Phi(\bs\mu)=0$.
\end{proof}

In typical  Scanning Magnetic Microscopy experiments, data consists of 
point-wise values of one component of the magnetic field taken on a plane not 
intersecting $S$. Of course, finitely many values do not
characterize the field, but it is natural to ask how one can 
choose  the measurement points to ensure that infinitely many of them would,
in the limit, determine $\bb{b}(\bs{\mu})$  uniquely.
We next provide a sufficient condition that such data (more generally,
data measured  on an
analytic surface which needs not be a plane)
determines the field in the complement of $S$.  The condition dwells on the
remark that a nonzero real analytic function on a connected
open subset of $\R^k$ 
has a zero set of Hausdorff dimension at most $k-1$. It is so because 
the zero set is locally a countable union of smooth (even real-analytic)
embedded submanifolds of strictly positive 
codimension, see \cite[thm 5.2.3]{KP}. This fact 
sharpens the property that the zero set of a nonzero 
real analytic function  in $\R^k$ has Lebesgue measure zero,
and will be used at 
places in the paper.  Using
local coordinates, it is immediately checked that the previous bound
on the Hausdorff dimension 
 remains valid when $\R^k$ is replaced by a smooth real-analytic 
manifold embedded in $\R^m$ for some $m>k$. 

We also need at this point a version of the Jordan-Brouwer separation theorem
for a connected, properly embedded  ({\it i.e.} 
complete but not necessarily compact) surface in $\R^3$, to the effect that 
the complement of such a surface has two connected components.
In the smooth case which is our concern here,
we give in Appendix \ref{app1} a short, differential topological argument 
for this result which we assume is known but for which we could not 
find a published reference. 

Hereafter, we will say that $E,F\subset\R^3$ are positively separated from each 
other if $\textrm{dist}(E,F)>0$, where $\textrm{dist}(E,F):=\inf_{x\in E,y\in F}|x-y|$.

\begin{lemma}
\label{Lem4} Let $S\subset \R^3$ be closed and suppose $\R^3\setminus S$ 
is connected and   contains a nonempty open half-cylinder of direction 
$v\in \R^3\setminus \{0\}$.  Furthermore, let $\mathcal{A} $ be a smooth
complete  and connected real analytic surface in $\R^3\setminus S$ that is positively separated from $S$ and such that $S$ lies entirely
within one of the two connected components of $\R^3\setminus \mathcal{A}$.
Let also 
$Q\subset\R^3\setminus S$ be such that the 
closure of  $Q\cap \mathcal{A}$  has 
Hausdorff dimension strictly greater than 1.
If $\bs{\mu}\in \mathcal{M}(S)^3$ is such that $v\cdot \bb{b}(\bs{\mu})$ 
vanishes on 
$Q\cap \mathcal{A}$,  then $\bs{\mu}$ is $S$-silent.
\end{lemma}

\begin{proof}
Suppose $\bs{\mu}\in \mathcal{M}(S)^3$ is such that $v\cdot \bb{b}(\bs{\mu})$ vanishes on $Q\cap\mathcal{A}$.
 As $v\cdot\bb{b}(\bs{\mu})$ is harmonic in $\R^3\setminus S$ it is 
real-analytic  there, and since $v\cdot\bb{b}(\bs{\mu})$  vanishes on the 
closure of   $Q\cap\mathcal{A}$  which has Hausdorff dimension $>1$  
it must vanish identically on $\mathcal{A}$. 

Observe now that $\R^3\setminus\mathcal{A}$ has two connected components 
(see Theorem~\ref{JB}), and let  $\mathcal{U}$ 
be the one not containing $S$. 
%
%
 Note, using \eqref{bDef1}, that if $\bs{\nu}\in \mathcal{M}(S)^3$ and $x \notin \text{supp}(\bs{\nu})$, then,
 \begin{align}\label{eqn:b_bound}
 	|\bb{b}(\bs{\nu})(x)| \le 4c \left(\text{dist}(x, \text{supp}(\bs{\nu}))\right)^{-3} \|\bs{\nu}\|_{TV}.
 \end{align}
 For $R>0$ let $\bs{\mu}_R:=\bs{\mu}\mathcal{b}B(0,R)$ and $\tilde{\bs{\mu}}_R:=\bs{\mu}-\bs{\mu}_R$.  Then
 $$\bb{b}(\bs{\mu})(x)=\bb{b}(\bs{\mu_R})(x)+\bb{b}(\tilde{\bs{\mu}}_R)(x),$$
and applying \eqref{eqn:b_bound} to $\bs{\mu}_R$ and $\tilde{\bs{\mu}}_R$ for $R$
large enough, using that $\mathcal{A}$ is positively separated from $S$,
we get that  $\limsup_{x\in U,|x|\to\infty}|\bb{b}(\bs{\mu})(x)|<\varepsilon$
for any $\varepsilon>0$, hence 
$\bb{b}(\bs{\mu})(x)\to 0$ as $x\to \infty$  in $\mathcal{U}$.
Since $v\cdot \bb{b}(\bs{\mu})$ vanishes on the boundary of $\mathcal{U}$,   we may use  the maximum principle to conclude  that $v\cdot \bb{b}(\bs{\mu})$ vanishes on $\mathcal{U}$ and therefore on $\R^3\setminus S$ as  the complement of $S$  is connected.

This implies that the magnetic potential $\Phi(\bs{\mu})$ is constant 
on every line segment parallel to $v$  not intersecting $S$. 
Now, $\R^3\setminus S$ contains a half-cylinder $\mathcal{C}$
of  direction $v$, 
and shrinking the latter if necessary we may assume it is positively 
separated from $S$. From \eqref{PhiDef} we get 
 \begin{align}\label{eqn:Phi_bound}
 	|\Phi(\bs{\nu})(x)| \le \frac{1}{4\pi} \left(\text{dist}(x, \text{supp}(\bs{\nu}))\right)^{-2} \|\bs{\nu}\|_{TV}
 \end{align}
 for $\bs{\nu}\in \mathcal{M}(S)^3$ and $x \notin \text{supp}(\bs{\nu})$, and we conclude that 
$\Phi(\bs{\mu})(x)$ 
goes to zero as $x\to \infty$ in $\mathcal{C}$.  Hence its value on each
half line contained in $\mathcal{C}$ is zero,
 so $\Phi(\bs{\mu})\equiv0$ 
in $\mathcal{C}$. Consequently it  vanishes identically
(thus also $\bb{b}(\bs{\mu})$) in  the connected open set
$\R^3\setminus S$, by real analyticity.
\end{proof}


The following example shows that  
Lemma~\ref{Lem4} needs not hold if 
$S$ is not contained in a single component of $\R^3\setminus \mathcal{A}$
or if $\mathcal{A}$ fails to be analytic.

\begin{exa}
	Let $S$ be equal to $\{e_3,-e_3\}$, $\bs{\mu}=(\delta_{e_3}+\delta_{-e_3})e_2$, $v=e_3$, and $\mathcal{A}=\{x_3=0\}$. 
	Then   $e_3\cdot\bs{b}(\bs\mu)$ is zero on $\mathcal{A}$
 but $\bs{\mu}$ is not $S$-silent.
Also, whenever $Q$ is a bounded subset of $\{x_3=0\}$ with
$\textrm{dim}_{\mathcal{H}} Q>1$,  
there is a closed $C^\infty$-smooth  surface $Z$ containing $Q$
such that $e_3$ and $-e_3$ lie in the same component of 
$\R^3\setminus Z$;
however, $Z$ cannot be analytic.
\end{exa}
Lemma~\ref{Lem4} does not hold either if 
$\R^3\setminus S$ is disconnected: for instance, if $S=\{x_3=0\}$, there is
a $\bs{\mu}$ of the form $\bb{h}d\mathcal{L}_2$ with $\bb{h}\in(\mathfrak{h}^1(\R^2))^3$, where $\mathfrak{h}^1(\R^2)$ is the real Hardy space on $\R^2$,
such that $\bb{b}(\bs{\mu})$ vanishes identically for $\{x_3<0\}$  but
not for $\{x_3>0\}$ \cite[Thm. 3.2]{IMP}. In particular, if we let $Q=\mathcal{A}=\{x_3=-1\}$, 
we have that $\bb{b}(\bs{\mu})$ vanishes on $Q$ but $\bs{\mu}$ is not 
$S$-silent.
However, if $\R^3\setminus S$ is not connected but has a connected component  $\mathcal{V}$ containing  a half-cylinder, then  by replacing $S$ with $\tilde S:=\R^3\setminus \mathcal{V}$ and selecting an appropriate $\mathcal{A}$, $Q$ and $v$, 
Lemma~\ref{Lem4} may be applied to the effect that $\bs{\mu}$ is $\tilde S$-silent whenever $v\cdot \bb{b}(\bs{\mu})$ vanishes on $Q$.  Thus, 
if each component $\mathcal{V}_i$ of $\R^3\setminus S$ contains a half-cylinder  and can be associated with suitable $\mathcal{A}_i$, $Q_i$ and $v_i$, and if $v_i\cdot \bb{b}(\bs{\mu})$ vanishes on $Q_i$ for all $i$, then $\bs{\mu}$ is $S$-silent. 

\begin{rmk}
If $S$ is as Lemma \ref{Lem4} and $Q\subset\R^3$ 
is positively separated from $S$
and has closure $\bar{Q}$ of Hausdorff dimension
$>2$, then the conclusion 
of the lemma still holds 
 and the proof is easier. In this case indeed, it follows 
directly from the hypothesis on $Q$ that
$v\cdot\bb{b}(\bs{\mu})$ is 
identically zero in  $\R^3\setminus S$ as soon as it vanishes
on $Q$, and the rest of the proof is as 
before. We shall not investigate 
this situation which, from the point of view of inverse problems, 
corresponds to the case where measurements of the 
field are taken in a volume rather than on a surface. 
Though more information can be gained this way, the experimental
and computational burden often becomes deterring.

\end{rmk}

\begin{lemma} \label{components} Let $S= S_0\cup S_1\subset \R^3$  for some disjoint closed sets $S_0$ and $S_1$. If
$\bs{\mu}\in\mathcal{M}(S)^3$ is $S$-silent, then for $i=0,1$  the restriction  $\bs{\mu}\mathcal{b}S_i$   is 
$S_i$-silent.  
\end{lemma}
 
\begin{proof}
Let $\bb b_0=\bb b(\bs{\mu}\mathcal{b}S_0)$ and $\bb b_1=\bb b(\bs{\mu}\mathcal{b}S_1)$.
Note that $\bb b_0$ and $\bb b_1$ are harmonic in $\R^3\setminus S_0$ and $\R^3\setminus S_1$ respectively.  
Also, as  $\bs{\mu}$ is $S$-silent, it holds that $\bb{b}_0(x)=-\bb{b}_1(x)$ for
$x\not \in S$.
Hence the function
$$\widetilde{\bb b} (x)=\begin{cases} \bb b_0(x) & x\in \R^3\setminus S_0,\\
-\bb b_1(x) & x\in \R^3\setminus S_1,
\end{cases}
$$
is harmonic on $\R^3$ (note that the two definitions agree on
$\R^3\setminus S$). 
Moreover, Proposition~\ref{LocIntPhi} implies that for every $\bb f\in (C_c^\infty(\R^3))^3$:
\begin{equation}\label{components:eq1}
|\langle T_\alpha \bb f,\widetilde{\bb b}\rangle| \leq |\langle T_\alpha \bb f, \bb b_0\rangle| + |\langle T_\alpha \bb f, \bb b_1\rangle| \rightarrow 0 \qquad \text{ as }\qquad |\alpha|\rightarrow\infty.
\end{equation}
Since $\widetilde{\bb b}$ is harmonic, the mean value property applied to \eqref{components:eq1} with  $\bb f=\varphi e_j$, where $\varphi$ is a
non-negative radially symmetric smooth function with support $\overline{\B}(0,1)$,  implies that  $\widetilde{\bb b}(x)$ vanishes as $x\to \infty$ and
therefore  is identically 0 by   Liouville's  theorem. 
Thus both $\bb b_0$ and $\bb b_1$ are zero on $\R^3\setminus S_0$ and $\R^3\setminus S_1$ respectively, and hence $\bs{\mu}\mathcal{b}S_0$ is $S_0$-silent and $\bs{\mu}\mathcal{b}S_1$ is $S_1$-silent.
\end{proof}

\subsection{Decomposition of divergence-free magnetizations and recovery of magnetizations with sparse support}\label{sparse3D:Sec}
A set $E\subset\R^2$ is said to be {\em 1-rectifiable} (e.g., see \cite[Def. 15.3]{Mattila}) if there exist Lipschitz maps $f_i:\R\to\R^n$, $i=1,2,...,$ such that
$$
\mathcal{H}_1\left(E\setminus\bigcup_{i=1}^\infty f_i(\R)\right)=0.
$$
A set $B\subset \R^n$ is {\em purely 1-unrectifiable} if $\mathcal{H}_1(E\cap B)=0$ for every 1-rectifiable set $E$.  Clearly a set of $\mathcal{H}_1$-measure zero is purely 1-unrectifiable. A purely 2-unrectifiable set is defined in the same way, only with $\mathcal{H}_2$ instead of $\mathcal{H}_1$. 

We call a Lipschitz mapping  $\bs{\gamma}:[0,\ell]\to \R^3$  a {\em  rectifiable curve} and let $\Gamma:=\bs{\gamma}([0,\ell])$ denote its image.
If  $\bs{\gamma}$   is an arclength parametrization of $\Gamma$; i.e., if $\bs{\gamma}$ satisfies 
\begin{equation}\label{arc}
\mathcal{H}_1(\bs{\gamma}([\alpha,\beta]))=\beta-\alpha, \qquad \forall[\alpha,\beta]\subset [0,\ell],
\end{equation}
then we call $\bs{\gamma}$  
  an {\em oriented rectifiable curve}. 
  By   Rademacher's Theorem (see \cite{EvaGar1992}),   $\bs{\gamma}$ is differentiable a.e. on $[0,\ell]$.  Furthermore, it follows from \eqref{arc} that $|\bs{\gamma}'(t)|=1$ a.e. on $[0,\ell]$.   For a given oriented rectifiable curve $ \bs{\gamma}$ we define $\mathbf{R}_{\bs{\gamma}}\in \mathcal{M}(S)^3$
  through the relation
\begin{equation}
\langle \mathbf{R}_{\bs{\gamma}},\mathbf{f}\rangle=\int_0^\ell \mathbf{f}(\bs{\gamma}(t))\cdot\bs{\gamma}'(t) dt,
\end{equation}
for $\mathbf{f}\in C_0(\R^3)^3$.   Alternatively, since $\mathbf{R}_{\bs{\gamma}}$ is absolutely continuous with respect to $\mathcal{H}_1$ we may consider the Radon-Nikodym derivative $\bs{\tau}$ of $\mathbf{R}_{\bs{\gamma}}$ with respect to
$\mathcal{H}_1$ and we remark that   $\bs{\tau}(\bs{\gamma}(t))=\bs{\gamma}'(t)$ for a.e. $t\in[0,\ell]$. Then, for a Borel set $B\subset \R^3$ we have
\begin{equation}\label{Rgamma}
 \mathbf{R}_{\bs{\gamma}}(B)=\int_B \bs{\tau} \, d (\mathcal{H}_1\mathcal{b}\Gamma ).
\end{equation}
We remark that if $B$ is purely 1-unrectifiable, then  $|\mathbf{R}_{\bs{\gamma}}|(B)=\mathcal{H}_1(B\cap\Gamma)=0$ and, furthermore,  Fubini's Theorem 
implies $\mathcal{L}_3(B)=0$.

  Let $\mathcal{C}\subset \mathcal{M}(S)^3$
denote the collection  of oriented rectifiable curves with topology inherited from $\mathcal{M}(S)^3$. 
Suppose $\div\bs\mu=0$ (as a distribution).  Smirnov \cite[Theorem A]{Smi94} shows that $\bs{\mu}$ can be decomposed into elements from $\mathcal{C}$. 
In particular, it can be proven that there is a positive  measure $\rho$ on $\mathcal{C}$ 
such that
\begin{equation}\label{Smi1}
\bs{\mu}(B)=\int \mathbf{R}(B)\ d\rho(\mathbf{R}),
\end{equation}
and 
\begin{equation}\label{Smi2}
|\bs{\mu}|(B)=\int |\mathbf{R}|(B)\ d\rho(\mathbf{R}),
\end{equation}
for any Borel set $B\subset\R^3$.  From the representation \eqref{Smi1} of a divergence-free magnetization, we immediately obtain the following lemma. 
\begin{lemma}\label{1unrecdiv}
Suppose $S\subset \R^3$ is closed and purely 1-unrectifiable.  If $\bs{\mu}\in \mathcal{M}(S)^3$ is divergence-free, then $\bs{\mu}=0$. 
\end{lemma}

\begin{theorem}\label{unrec3D}
Suppose $S\subset \R^3$ is a closed, slender set.  
If $\bs{\mu}\in\mathcal{M}(S)^3$ 
has support that is purely 1-unrectifiable and  $\bs{\nu}\in\mathcal{M}(S)^3$ is $S$-equivalent to $\bs{\mu}$, then $\|\bs{\nu}\|_{TV}>\|\bs{\mu}\|_{TV}$
unless $\bs{\nu}=\bs{\mu}$.
\end{theorem}
\begin{proof}
Since $\bs{\theta}:=\bs{\nu}-\bs{\mu}$ is $S$-silent,   Theorem~\ref{divM} implies  $\div\bs{\theta}=0$ and so $\bs{\theta}$ can be represented in the form 
\eqref{Smi1}, where \eqref{Smi2} holds.  Since the support of $\bs{\mu}$ is   purely 1-unrectifiable, it follows from \eqref{Smi2} and the remark after
\eqref{Rgamma}
that the measures $\bs{\mu}$ and $\bs{\theta}$ are mutually singular.  Thus,
$\|\bs{\nu}\|_{TV}=\|\bs{\mu}\|_{TV}+\|\bs{\theta}\|_{TV}>\|\bs{\mu}\|_{TV}$ unless $\bs{\nu}=\bs{\mu}$.
\end{proof}

\begin{exa}\label{Snotslender}
Recall from Example~\ref{balldipole} that if $S=\overline{\B}(0,1)$, 
the magnetizations $\bs{\mu}$ modeling a uniformly magnetized ball and 
$\bs{\nu}$ which
is a point dipole at $0$, with the same net moment as $\bs{\mu}$, 
are $S$-equivalent.  Moreover,  it is easy to verify that $\|\bs{\mu}\|_{TV}=\|\bs{\nu}\|_{TV}=|\langle\nu\rangle|$.   Since the support of $\bs{\nu}$ is 
a single  point, it is purely 1-unrectifiable, hence
the assumption that 
$S$ is a slender set cannot be eliminated from Theorem \ref{unrec3D}.  
\end{exa}

The previous example entails that total variation minimization is not sufficient alone to distinguish  magnetizations with purely 1-unrectifiable support among all equivalent magnetizations supported on $S$, when $S$ is not slender.  However, as the following result shows, the recovery problem for general $S$ 
has at most one solution when restricted to magnetizations whose support  
is purely 1-unrectifiable {\em and} has finite $\mathcal{H}_2$ measure.   

To see this, we shall need a consequence of the Besicovitch-Federer Projection Theorem~\cite[Thm 18.1]{Mattila}; namely that the complement of a closed, purely 2-unrectifiable set with finite $\mathcal{H}_2$ measure is connected.  
We will
also need the fact that a purely 1-unrectifiable set  is  purely 2-unrectifiable.  We are confident these facts are known (e.g., see the introduction in \cite{DavSem}), but since we have not explicitly found proofs in the literature, we provide outlines of the arguments.   

With regard to the first fact, let $F\subset \R^3$ be a closed, purely 2-unrectifiable set with finite $\mathcal{H}_2$ measure  and suppose   ${\mathbb B}(x,r)$,  ${\mathbb B}(y,r)$ are disjoint balls in    $\R^3\setminus F$.  By the Besicovitch-Federer Projection Theorem there exists a plane $P$ such that the intersection of the orthogonal projections of these balls onto $P$ minus the orthogonal projection of $F$ onto $P$ is nonempty and therefore the balls can be joined by a line segment not intersecting $F$.  

As to the second fact, it follows from \cite[Lemma~3.2.18]{FedererBook} that it is enough for a set $F$ to be purely 2-unrectifiable that $\mathcal{H}_2(F\cap \psi(K))=0$ for any compact set $K\subset \R^2$ and any bi-Lipschitz mapping $\psi:K\to \R^3$.   Since bi-Lipschitz maps preserve unrectifiability we 
may restrict our considerations to $\R^2$ where the result follows easily from Fubini's theorem.  

As a consequence of these facts,  the complement of a closed, purely 1-unrectifiable set with  finite $\mathcal{H}_2$ measure must be connected.

\begin{corollary}\label{Cor-unrec3D}
Suppose $S$ is a closed, proper subset of $\R^3$ and that $\bs{\nu}\in\mathcal{M}(S)^3$ has purely 1-unrectifiable support 
of finite $\mathcal{H}_2$ measure.  If $\bs{\mu}\in\mathcal{M}(S)^3$ is $S$-equivalent to $\bs{\nu}$ but not equal to $\bs{\nu}$, then the support of $\bs{\mu}$ is not a purely 1-unrectifiable set with finite $\mathcal{H}_2$ measure. 
\end{corollary}
\begin{proof}
Suppose $\bs{\mu}\in\mathcal{M}(S)^3$ is $S$-equivalent to $\bs{\nu}$  and has support that is purely 1-unrectifiable with finite $\mathcal{H}_2$ measure.  
Then the support $\tilde S$ of $\bs{\mu}-\bs{\nu}$ is also purely 1-unrectifiable with finite $\mathcal{H}_2$ measure.  Therefore, its complement is connected and thus $\tilde S$ is slender.

Moreover, $\bs{\mu}-\bs{\nu}$ is $S$-silent, hence its field vanishes on
the nonempty open set $\R^3\setminus S$, and since $\tilde S$ is closed
with $\mathcal{L}_3(\tilde S)=0$ (because it is slender), the field
must vanish on a nonempty open subset of $\R^3\setminus\tilde S$. 
Since $\tilde S$ has connected complement, we conclude that
$\bs{\mu}-\bs{\nu}$ is $\tilde S$-silent, by real analyticity. Consequently, it is divergence-free by Theorem~\ref{divM} and hence,  by Lemma~\ref{1unrecdiv}, 
$\bs{\mu}-\bs{\nu}$ is the zero measure.
\end{proof}

Corollary \ref{Cor-unrec3D} applies in particular if $\bs{\nu}$ 
is a finite sum of point dipoles. However, in view of Example \ref{balldipole},
it does not apply in general to a convergent series of point dipoles.

\subsection{The net moment of silent magnetizations}
  
Our next result shows that, under certain assumptions on their support, 
silent measures have vanishing moment:

\begin{lemma}\label{MomZero}
	Let $S\subset\R^3$ be a closed set and $\bs{\mu}\in\mathcal{M}(S)^3$ be $S$-silent.  
	Assume that one of the following conditions is satisfied:
	\begin{enumerate}
		\item $S$ is compact,
		\item $\div\bs{\mu}=0$.
	\end{enumerate}
	Then the net moment $\langle \bs{\mu}\rangle=0$.  
\end{lemma}
\begin{proof}
	Fix $i\in\{1,2,3\}$.
	Let $\phi\in C^\infty_0(\R^3)$ be supported in $\B(0,2)$, $\phi(x)=x_i$ on $\B(0,1)$ and for any $n>0$ let $\phi_n(x):=n\phi(x/n)$.
	Note that $\|\grad\phi\|_\infty=\|\grad\phi_n\|_\infty$, that
$\phi_n$ is supported in $\B(0,2n)$, and for $x\in\overline{\B}(0,n)$ that $\phi_n(x)=x_i$, $\grad\phi_n(x)=e_i$ and $\Delta\phi_n(x)=0$.
	
	If $S$ is compact  take $n>0$ such that $S\subset\B(0,n)$.
	Then $\langle\grad\phi_n,\bs{\mu}\rangle=\langle \bs{\mu}\rangle_i$, the $i$-th component of the  moment of $\bs{\mu}$.
	Since $\bs{\mu}$ is $S$-silent $\Phi(\bs\mu)$ is constant in each connected component of $\R^3\setminus S$,  and applying \eqref{inegPhig} with $E=\B(0,R)\setminus\B(0,n)$ we conclude on letting $R\to\infty$ that 
$\Phi(\bs\mu)\equiv0$ in the unbounded component.
	Thus, $\Phi(\bs{\mu})$ is supported on $\B(0,n)$ and since $\Delta\Phi(\bs{\mu})=\div \bs{\mu}$ by Proposition~\ref{LocIntPhi}, we obtain:
	\begin{align*}
	\langle \bs{\mu}\rangle_i = \langle\grad\phi_n,\bs{\mu}\rangle 
	= \langle\phi_n,\div\bs{\mu}\rangle = \langle \phi_n,\Delta\Phi(\bs{\mu})\rangle = \langle\Delta\phi_n,\Phi(\bs{\mu})\rangle=0.
	\end{align*}
	Therefore, taking $i=1,2,3$, we get that $\langle\bs{\mu}\rangle=0$, as announced.
	
	Assume next that $\div\bs{\mu}=0$. For any integer $m>0$, let $D_m:=\B(0,2^{m+1})\setminus\B(0,2^m)$ and $M_m:=|\bs{\mu}|(D_m)$.
	Because $\sum_m M_m \leq \|\bs{\mu}\|_{TV} < \infty$, we have that
\begin{equation}
\label{liman} 
	|\langle \grad\phi_{2^m}|_{D_m} , \bs{\mu} \rangle| \leq M_m\|\grad\phi\|_\infty \to 0\quad \textrm{as}\quad m\to\infty.
\end{equation}
	Now, let $\mathbf{E}_i$ be the constant function equal to $e_i$ on $\R^3$.
	By \eqref{liman}, we see that 
	$$
	\lim_{m\to\infty}\langle\grad\phi_{2^m},\bs{\mu}\rangle 
	= \lim_{m\to\infty}\langle\grad\phi_{2^m}|_{\B(0,2^{m})}, \bs{\mu}\rangle
	= \langle \mathbf{E}_i, \bs{\mu}\rangle = \langle \bs{\mu}\rangle_i,
	$$
	and since  $\langle\grad\phi_n,\bs{\mu}\rangle=\langle\phi_n,\div\bs{\mu}\rangle=0$ for each $n>0$, by our assumption, we conclude that
	$\langle\bs{\mu}\rangle=0$, as desired.
\end{proof}

%

Assumptions (a) or (b) cannot be dropped in Lemma \ref{MomZero}, for
it is not sufficient that a magnetization be $S$-silent for its net 
moment to vanish,  as shown by the following example.

\begin{exa} 
\label{momsil}
Consider the case where $S=\R^3\setminus\B(0,R)$ 
and let $\bs{\mu}=v\mathcal{H}_2\mathcal{b}\sph(0,R)$
where $v\in \R^3\setminus\{0\}$. 
The density of $\bs{\mu}$ with respect to
$\mathcal{H}_2\mathcal{b}\sph(0,R)$ is the
constant map $\bb{v}:\sph(0,R)\to\R^3$
given by $\bb{v}(x)= v$, which is the trace on $\sph(0,R)$ of the gradient of
the function $x\mapsto v\cdot x$ which is harmonic 
on a neighborhood of $\overline{\B(0,R)}$,
hence $\bb{v}$ {\it a fortiori} belongs to the Hardy space 
$\mathcal{H}^2_{+,R}$ of harmonic gradients in $\B(0,R)$.
Therefore $\bs{\mu}$ is silent in that ball \cite[Lemma 4.2]{BaGe},
and still $\langle \bs{\mu}\rangle=4\pi v$.
Integrating this example  over $R\in[1,\infty)$ against the weight 
$1/R^4$ further shows that the $\R^3$-valued measure
$d\bs{\nu}(x)={\bb v}|x|^{-4}\chi_{\{|x|\geq1\}}(x)d\mathcal{L}_3(x)$,
is silent in the ball
$\B(0,1)$ but has $\langle\bs{\nu}\rangle=4\pi v$. 
This provides us with an example of a (non-compactly supported)
measure with non-zero total moment 
which is silent in the complement of its support. 
\end{exa}


\subsection{Total variation and unidirectional magnetizations}\label{uni3D:Sec}

For $\bs{\mu}\in\mathcal{M}(S)^3$ we can write
$d\bs{\mu}=\bb{u}_{\bs{\mu}}d{|\bs{\mu}|}$,  therefore
   the Cauchy Schwarz inequality yields that
\begin{equation}\label{mom1}
|\langle \bs{\mu}\rangle|^2= \int \langle \bs{\mu}\rangle \cdot \bb{u}_{\bs{\mu}} \, d{|\bs{\mu}|}\le |\langle \bs{\mu}\rangle| \, \|\bs{\mu}\|_{TV},
\end{equation}
where equality holds if and only if either $\langle \bs{\mu}\rangle=0$ or $\bb{u}_{\bs{\mu}}=\langle \bs{\mu}\rangle/|\langle \bs{\mu}\rangle|$ a. e. with respect to $|\bs{\mu}|$.  We say that $\bs{\mu}$ is {\em uni-directional} if $\bb{u}_{\bs{\mu}}$ is constant  a.e. with respect to $|\bs{\mu}|$ (note
that the zero magnetization is uni-directional).  Thus,   \eqref{mom1} implies the following:
\begin{lemma}\label{lemuni1}
If $\bs{\mu}\in\mathcal{M}(S)^3$, then $|\langle \bs{\mu}\rangle |\le \|\bs{\mu}\|_{TV}$
with equality if and only if $\bs{\mu}$ is uni-directional.
\end{lemma}


%
 
 
 We call a magnetization {\em uni-dimensional} if it is the difference of two uni-directional magnetizations. 
 The next lemma states that   a    uni-dimensional magnetization which is divergence-free   must be the zero magnetization. 
 \begin{lemma}\label{unidiv0}
 If  $\bs{\mu}\in\mathcal{M}(\R^3)^3$ is uni-dimensional and $\div \bs{\mu}=0$, then $\bs{\mu}=0$.
 \end{lemma}
 \begin{proof} Suppose $\bs{\mu}\in\mathcal{M}(\R^3)^3$ is uni-dimensional and divergence-free.   Then $\bs{\mu}=\mu v$ for some $v\in \R^3$ and $\mu\in \mathcal{M}(\R^3)$, with 
 	 		  $0=\div (\mu v) =v\cdot \grad\mu $.  A standard argument (see below) shows that $\mu$ is translation invariant with respect to any vector parallel to $v$  and therefore $\mu$ is finite only if it is zero.

 		Without loss of generality we may assume that $v=(1,0,0)$.
	To see the translation invariance of  $\mu$, take any $f\in\mathcal{C}_c^\infty(\R^3)$ and let $\tilde{f}$ be a translation of $f$ in the $x_1$ direction. 
 		Then $f-\tilde f=\partial_{x_1}g$ with $g\in\mathcal{C}_c^\infty(\R^3)$ defined by:
 		$$
 		g(x_1,x_2,x_3):=\int_{-\infty}^{x_1}(f-\tilde f)(y,x_2,x_3) d\mathcal{L}_1(y),
 		$$
 		and so $\mu(f-\tilde f)=\langle\partial_{x_1}g,\mu\rangle =-\langle g,\partial_{x_1}\mu\rangle=0$.	%
 	%
 \end{proof}

\begin{theorem}  
\label{UnidirThm} Let $S=\bigcup_{i=1}^n S_i$ for some disjoint closed sets $S_1, S_2, \ldots, S_n$ in $\R^3$ and suppose that $S$ is either compact or slender.  Let  $\bs{\mu}\in \mathcal{M}(S)^3$  be such that $\bs{\mu}_i:=\bs{\mu}\mathcal{b}S_i$ is uni-directional for $i=1, 2, \ldots, n$.    

If $\bs{\nu}\in \mathcal{M}(S)^3$   is $S$-equivalent to $\bs{\mu}$, then $\bs{\nu}_i:=\bs{\nu}\mathcal{b}S_i$ and $\bs{\mu}_i$ are $S_i$-equivalent for $i=1, 2, \ldots, n$, moreover
\begin{equation}
\label{uniminTV}  \|\bs{\mu}\|_{TV}\le  \|\bs{\nu}\|_{TV},
\end{equation}
with equality in \eqref{uniminTV} if and only if      $\bs{\nu}_i$ is uni-directional in the same direction as $\bs{\mu}_i$ for $i=1, 2, \ldots, n$.
Furthermore, if $S$ is slender and equality holds in \eqref{uniminTV}, then $\bs{\mu}=\bs{\nu}$.
\end{theorem}

\begin{proof}
Since $\bs{\mu}$ and $\bs{\nu}$ are $S$-equivalent, their difference $\bs{\tau}:=\bs{\nu}-\bs{\mu}$ is $S$-silent.  
In addition, it follows from Theorem~\ref{divM} that if $S$ is slender
then $\div\bs{\tau}=0$.
By Lemma~\ref{components},
the restriction $\bs{\tau}_i:=\bs{\tau}\mathcal{b}S_i$ is $S_i$-silent and thus $\bs{\mu}_i$ and $\bs{\nu}_i$ are $S_i$-equivalent.
Since either $S$ is compact or $\div\bs{\tau}=0$, the same is true of each
$S_i$, $\bs{\tau}_i$ and we can use Lemma~\ref{MomZero} to obtain that
$\langle\bs{\mu}_i\rangle=\langle\bs{\nu}_i\rangle$ for $i=1, 2, \ldots, n$. Then
\begin{equation}\label{numuTV}
 \|\bs{\nu}\|_{TV}=\sum_{i=1}^n\|\bs{\nu}_i\|_{TV}\ge \sum_{i=1}^n |\langle\bs{\nu}_i\rangle |=\sum_{i=1}^n |\langle\bs{\mu}_i\rangle |= 
 \sum_{i=1}^n\|\bs{\mu}_i\|_{TV}=\|\bs{\mu}\|_{TV}, 
\end{equation}
where the next to last equality follows from the uni-directionality of $\bs{\mu}_i$.   By Lemma~\ref{lemuni1}, equality holds in \eqref{numuTV} if and only if each $\bs{\nu}_i$ is uni-directional, and it must have the direction of $\bs{\mu}_i$ since their moments agree. In particular $\bs{\tau}$ is then
unidimensional, hence if in addition $S$ is slender so that
$\div\bs{\tau}=0$, we get from Lemma~\ref{unidiv0} 
that equality holds in  \eqref{uniminTV} only when
$\bs{\mu}=\bs{\nu}$.
\end{proof}

Note that if $S$ is not slender, 
$\bs{\mu}$ and $\bs{\nu}$ in the previous theorem can be different, 
even when equality holds in  \eqref{uniminTV}. Indeed,
for $n=1$ already, Example~\ref{balldipole} yields such a situation.

\section{Magnetization-to-field operators}
\label{MtoFop}

Let $S\subset \R^3$ be closed  and $Q\subset\R^3\setminus S$ be  compact.
For $\bs{\mu}\in \mathcal{M}(S)^3$ and
$v$ a unit vector in $\R^3$,  the component  of the magnetic field
$\bb{b}(\bs{\mu})$ in the direction $v$ at $x\not \in S$ 
is given,  in view of  \eqref{bDef1}, by
\begin{equation}\label{b3K}
b_v(\bs{\mu})(x):=v\cdot\bb{b}(\bs{\mu})(x)  =-\frac{\mu_0}{4\pi}\int \bb K_v(x-y)\cdot \, d\bs{\mu}(y),
\end{equation}
where 
\begin{equation}
\label{noyauK}
\bb K_v(x)=\frac{v}{|x|^3}  -3  x\frac{v\cdot x}{|x|^5}=  \grad \left( \frac{v\cdot x}{|x|^3}\right).
\end{equation}
Consider  a finite, positive Borel measure  $\rho$
with support contained in $Q$  and let $A:
\mathcal{M}(S)^3\to L^2(Q,\rho)$ be the operator
defined by
\begin{equation}\label{Adef}
A(\bs{\mu})(x):= b_v(\bs{\mu})(x), \qquad x\in Q.
\end{equation}  
Since $\bb{K}_v$ is continuous on $\R^3\setminus \{0\}$ and  $Q$ and $S$ are positively separated, it follows that $b_v$  
is continuous on $Q$ and consequently
$A$ does indeed map  $\mathcal{M}(S)^3$ into $L^2(Q,\rho)$. 
 
%

If $\Psi\in L^2(Q,\rho)$, then using Fubini's Theorem and \eqref{b3K} we have
that
\begin{equation}
\langle \Psi, A(\bs{\mu})\rangle_{ L^2(Q,\rho)}=-\frac{\mu_0}{4\pi}\iint \Psi(x) \bb K_v(x-y)\cdot \, d\bs{\mu}(y)\ d\rho(x)=
\langle A^*( \Psi), \bs{\mu}\rangle,
\end{equation}
where for $x\in S$  the adjoint operator   $A^*$ is given by
\begin{equation}
  \label{b3*W}
  A^*(\Psi)(x):=-\frac{\mu_0}{4\pi}\int \Psi(y)\bb K_v(x-y)\ d\rho(y),\qquad x\in S.
\end{equation}
In view of \eqref{noyauK}, a compact way of re-writing \eqref{b3*W} is
\begin{equation}
  \label{b3*Wbis}
  A^*(\Psi)(x):=-\mu_0 \grad(\grad U^{\rho,\psi}\cdot v)(x),\qquad
U^{\rho,\psi}(x)=-\frac{1}{4\pi}\int \frac{\Psi(y)}{|x-y|}d\rho(y).
\end{equation}

Since $Q$ and $S$ are positively separated it follows as in the proof of Lemma~\ref{Lem4} that $A^*(\Psi)\in C_0(S)^3$ and thus 
  $A^*:(L^2(Q,\rho))^*\sim L^2(Q,\rho)\to C_0(S)^3\subset(\mathcal{M}(S)^3)^*$.   

From the point of view of the inverse problem described in Section \ref{intro},
consisting of recovering $\bs{\mu}$ from measurements of the field of the potential of its divergence in the direction $v$, one should think of
$S$ as a set 
{\it a priori} containing the support of the sources to be
recovered, and of $Q$ as the set on which the component of the field in the 
direction $v$ is measured.
In the proposition below, we single
out two additional assumptions
on the pair $(S,Q)$, namely:
	\begin{enumerate}
		\item[(I)]$\R^3\setminus Q$ is connected,
$\mathcal{L}_3(Q)=0$ and $\mathcal{H}_d(S)>2$.
		\item[(II)] $\R^3\setminus Q$ is connected,
$\mathcal{L}_3(Q)=0$ and there is a smooth complete real analytic surface
$\mathcal{B}$ such that $Q$ lies in a single connected component 
of $\R^3\setminus \mathcal{B}$, while $\mathcal{H}_d(S\cap\mathcal{B})>1$.
	\end{enumerate}

\begin{prop}
\label{Acomp}
Let $S\subset \R^3$ be closed, $Q\subset\R^3\setminus S$  be  compact,  $\rho$ be  a finite, positive Borel measure  
with support contained in $Q$, and $v$ a unit vector in $\R^3$. 
\begin{enumerate}
\item The operator $A:
\mathcal{M}(S)^3\to L^2(Q,\rho)$ defined in \eqref{Adef} is compact.
\item  Each function in the
range of $A^*$ is the restriction to $S$ of a real-analytic 
$\R^3$-valued function on $\R^3\setminus Q$. 
\item  If either assumption \mbox{\rm (I)} or \mbox{\rm (II)} holds,
then $A^*$ is injective, hence $A$ has dense range.
\item If $Q,S,v$ satisfy the assumptions  of
Lemma \ref{Lem4} and the support of $\rho$ contains $Q\cap\mathcal{A}$,
then every element in the kernel of $A$ is $S$-silent.
\end{enumerate}
\end{prop}
\begin{proof}
	
Let $h:=\text{dist}(S,Q)>0$.
Outside an open ball of radius $h$, the kernel $\bb{K}_v$ and its  first order derivatives  are bounded, say  by some constant $C$.
Thus, if $\bs{\mu}_n$ is a 
sequence in the unit ball of $\mathcal{M}(S)^3$, then $|b_u(\bs{\mu}_n)|$ and 
its partial derivatives 
are bounded by $C$ on $Q$. 
Therefore $b_v(\bs{\mu}_n)$ is a uniformly bounded 
family of equicontinuous
functions on the compact set $Q$, hence it
is relatively compact in the uniform topology by Ascoli's Theorem.
{\it A fortiori}, this family is relatively compact in $L^2(Q,\rho)$.
Besides, since $\bb{K}_v$ is a harmonic vector field in $\R^3\setminus\{0\}$, 
differentiating \eqref{b3*W} under the integral sign shows
that the components of $A^*(\Psi)$ are harmonic in $\R^3\setminus Q$,
thus, {\it a fortiori} real analytic.

To see that $A$ has dense range if either (I) or (II) is satisfied, 
we  prove that $A^*$ is injective in this case. 
For this, assume that $A^*\Psi=0$ for some $\Psi\in L^2(Q,\rho)$ and let us 
show that $\Psi$ is zero $\rho$-a.e. 
Assume first that (II) holds, and consider the $\R^3$-valued function
$$
\bb D(x) =  \frac{1}{4\pi}\int \Psi(y)\frac{x-y}{|x-y|^3}\ d\rho(y),\qquad
x\in\R^3\setminus Q.
$$ 
Arguing as we did to get \eqref{DelPhi} and observing that $L^2(Q,\rho)\subset L^1(Q,\rho)$
since $\rho$ is finite, we find that $\bb D$ extends to a locally integrable function on $\R^3$ with
$\div\bb D= \Psi d\rho$ as distributions.
Note that $\grad(\bb D\cdot v)$ is a harmonic vector field on $\R^3\setminus Q$ 
which is equal to $-A^*\Psi/\mu_0$ on $S$, hence it vanishes there.
Since $S\cap\mathcal{B}$ has Hausdorff dimension strictly greater than 1 
and $\mathcal{B}$ is real analytic, 
it holds that $\grad(\bb D\cdot v)$ vanishes on $\mathcal{B}$.
Moreover, as $\bb K_v$ vanishes at infinity and $\Psi$ 
has compact support,  $\grad(\bb D\cdot v)$ vanishes at infinity as well.
Thus, by the maximum principle, it must vanish in the component of
$\R^3\setminus\mathcal{B}$ which does not contain $Q$, 
therefore also in $\R^3\setminus Q$ by real analyticity and since 
it is connected. This means that 
$\bb D\cdot v$ is constant in $\R^3\setminus Q$, and it is in fact 
identically zero because it is clear from the 
compactness of $Q$ that $\bb D$ vanishes at infinity.
Now, it holds that  $\bb D=\grad (N*(\psi d\rho))$, where 
$N(y)=-1/(4\pi|y|)$  is the Newton kernel already used in the proof of 
Proposition \ref{LocIntPhi}. Hence, $N*(\psi d\rho)$ must be constant on half lines 
parallel to $v$ contained in $\R^3\setminus Q$, and since it vanishes at 
infinity while $Q$ is compact we find that  $N*(\psi d\rho)$ is 
identically zero in  $\R^3\setminus Q$. Now, 
being the Newton potential of a finite measure, $N*(\psi d\rho)$ is a
locally integrable function and, since 
$\mathcal{L}_3(Q)=0$, we just showed that it is zero almost everywhere. 
Hence it is the zero distribution, and so is its 
weak Laplacian $\psi d\rho$. Consequently $\psi$ is zero $\rho$-a.e., as 
desired. If (I) holds instead of (II) the proof of (c) is similar but easier, 
because we conclude directly that  $\grad(\bb D\cdot v)=0$ on 
$\R^3\setminus Q$, since it is harmonic there and vanishes on $S$ which has 
Hausdorff dimension strictly greater than 2.  Finally, to prove (d), observe that if $A(\bs{\mu})=0$ a.e. 
with respect to $\rho$, then by continuity $A(\bs{\mu})=0$ on the support of $\rho$ and so on $Q\cap\mathcal{A}$.
Thus,  by Lemma \ref{Lem4}  
$ \bs{\mu} $ is $S$-silent whenever $ \bs{\mu}$ is in the kernel of $A$.
\end{proof}


\begin{rmk}
\label{rmkdens}
Note that if $Q$ contains a nonempty open ball $B$ 
(then of course $\mathcal{L}_3(Q)\neq0$ and neither (I) nor (II) 
is satisfied), and if moreover the 
support of $\rho$ contains $B$, then the image of $A$ is not dense in $L^2(Q,\rho)$ since it consists of functions harmonic in $B$.
\end{rmk}

The density of
$\textrm{ran} A$ plays no role in the forthcoming results, 
so from the sheer mathematical point of view
we may  forget about  Proposition \ref{Acomp} (c).
However, it is 
more often satisfied than not. For instance, in paleomagnetism, $Q$ would 
be a rectangular region in a plane and
$S$ a rock sample which is either volumic 
(then (I) is met) or sanded down to a thin slab (then (II) is met, with 
$\mathcal{B}$ a plane). In practice, if $A$ has dense range
then the data can be explained arbitrary well in terms of similarity
between the measured and modeled fields,
at the cost of proposing a model for $\bs{\mu}$ which is unreasonably large 
and therefore non-physical, 
see Theorem \ref{lamto0Thm} point (b). This phenomenon is typical of 
ill-posed  problems and a compelling reason why regularization is needeed.
That is, one must trade-off between the precision of the model against
available  data and its physical relevance, {\it e.g.} the regularization parameter $\lambda$ should not be made too small in  Theorem \ref{lamto0Thm}.


\section{Regularization by penalizing the total variation}
\label{regu}

In this section, we consider the inverse magnetization problem of recovering 
$\bs{\mu}\in\mathcal{M}(S)^3$ from the knowledge of $A(\bs{\mu})$, where $A$ is the operator defined in \eqref{Adef}. We will
study  the regularization scheme EP-2, based on \eqref{crit0}, 
that penalizes the  total variation of the 
candidate approximant, and prove that solutions to EP-2 exist and are 
necessarily ``localized'', in the sense that their 
support has dimension at most 
2 if $S$ has nonempty interior in $\R^3$
and dimension at most 1 if 
$S$ is contained in some unbounded 
 analytic surface where it has nonempty interior.
The existence of a solution to EP-2, as well as the optimality 
condition given in Theorem \ref{VPThm}, fall under the scope of
\cite[prop. 3.6]{BrePikk} and could just have been referenced. We nevertheless 
provide a proof, partly because it 
may be interesting in its own right as it is independent from the 
Fenchel duality used in \cite{BrePikk}, 
but mainly  because we want to discuss non-uniqueness in a specific manner.
We conclude this section with
a `consistency' result showing that solutions to EP-2 approach those of EP-1, 
in the limit that the regularization parameter $\lambda$ and the 
(additive) perturbation on the data vanish in a controlled manner. 
Our account of this regularization theory is  new inasmuch as  it 
includes the asymptotic behavior  
of total variation measures of the solutions, 
and deals with narrow convergence (not just weak-*).

Hereafter, as in Section \ref{MtoFop},
we let $S\subset \R^3$ be closed, $Q\subset\R^3\setminus S$  be  compact,  
$\rho$  be a finite, positive Borel measure 
supported in $Q$, and $v$ a unit vector in $\R^3$. The operator $A$ is 
then defined by \eqref{Adef}.
 For $\bs{\mu}\in\mathcal{M}(S)^3$, $f\in L^2(Q,\rho)$, and $\lambda>0$, we recall from 
\eqref{defcrit0} the definition of $\mathcal{F}_{f,\lambda}$:
\begin{equation}
\label{defcrit}
\mathcal{F}_{f,\lambda}(\bs{\mu}):=
\|f-A\bs{\mu}\|_{L^2(Q,\rho)}^2+\lambda \|\bs{\mu}\|_{TV},
\end{equation}
and from \eqref{crit0} that  $\bs{\mu_\lambda}\in \mathcal{M}(S)^3$ denotes a minimizer of  $\mathcal{F}_{f,\lambda}$ whose existence is proved in Theorem~\ref{VPThm} below; {\it i.e.}, 
\begin{equation}
\label{crit}
\mathcal{F}_{f,\lambda}(\bs{\mu_\lambda})=\inf_{\bs{\mu}\in\mathcal{M}(S)^3} \mathcal{F}_{f,\lambda}(\bs{\mu}).\end{equation}

\begin{theorem}\label{VPThm}
Notation and assumptions being as above, 
given $f\in L^2(Q,\rho)$ there is a solution  to (\ref{crit}). 
A $\R^3$-valued measure $\bs{\mu}_\lambda\in\mathcal{M}(S)^3$ is 
such a solution if and only if:
\begin{equation}
\label{CP}
\begin{array}{ll}
A^*(f-A\bs{\mu}_\lambda)&=\frac{\lambda}{2} \bb{u}_{\bs{\mu_\lambda}} 
\qquad |\bs{\mu}_\lambda|\text{\rm -a.e. and}\\
\left |A^*(f-A\bs{\mu}_\lambda)\right|&\leq \frac{\lambda}{2}\quad\text{\rm everywhere on } S.
\end{array}
\end{equation}
Moreover, $\bs{\mu}'_\lambda\in\mathcal{M}(S)^3$ is another solution 
if and only if:
\begin{enumerate}
\item $A(\bs{\mu}'_\lambda-\bs{\mu}_\lambda)=0$,
\item there is a $|\bs{\mu}_\lambda|$-measurable non-negative function $g$
and a positive measure $\nu_s\in\mathcal{M}(S)$, singular 
to $|\bs{\mu}_\lambda|$ and
supported on  the set  $\left\{x\in S \colon  \left|A^*(f-A\bs{\mu}_\lambda)(x)\right|=\lambda/2 \right\}$, such that
\begin{equation}
\label{equivsoldec}
d\bs{\mu}'_\lambda=gd\bs{\mu}_\lambda+2\frac{A^*(f-A\bs{\mu}_\lambda)}{\lambda}d \nu_s.
\end{equation}
\end{enumerate}
\end{theorem}
\begin{proof}Fix $\lambda>0$ and let $\bs{\mu}_n$ a minimizing sequence for the right hand side of (\ref{crit}).
By construction $\|\bs{\mu}_n\|_{TV}$ is bounded, hence we can find a subsequence 
that converges weak-$*$ to some $\bs{\mu}_\lambda$, by the Banach-Alaoglu 
Theorem. Renumbering if necessary,
let us denote this subsequence by $\bs{\mu}_n$ again. The Banach-Alaoglu
Theorem also entails that
\begin{equation}
\label{varl}
\|\bs{\mu}_\lambda\|_{TV} \leq\liminf_n \|\bs{\mu}_n\|_{TV}.
\end{equation}
Moreover, since $A$ is compact, 
$f-A(\bs{\mu}_n)$ converges to  $f-A(\bs{\mu}_\lambda)$ in $L^2(Q,\rho)$, hence
\begin{equation}
\label{convL2}
\|f-A(\bs{\mu}_\lambda)\|_{L^2(Q,\rho)}=\lim_n\|f-A(\bs{\mu}_n)\|_{L^2(Q,\rho)}.
\end{equation}
Because $\bs{\mu}_n$ is minimizing, it now follows from (\ref{varl})
and (\ref{convL2}) that $\bs{\mu}_\lambda$ meets (\ref{crit}) and that
(\ref{varl}) is both an equality and a true limit.

Let now $\bs{\nu}\in\mathcal{M}(S)^3$ be absolutely continuous with respect to
$|\bs{\mu}_\lambda|$ with Radon-Nykodim derivative 
$\bs{\nu}_a\in (L^1(\bs{\mu}_\lambda))^3$; that is to say:
$d\bs{\nu}=\bs{\nu}_ad|\bs{\mu}_\lambda|$.

We evaluate $\mathcal{F}_{f,\lambda}(\bs{\mu}_\lambda+t\bs{\nu})$ for 
small $t$.
On the one hand,
\begin{equation}
\label{expq}
\|f-A(\bs{\mu}_\lambda+t\bs{\nu})\|_{L^2(Q,\rho)}^2=
\|f-A(\bs{\mu}_\lambda)\|_{L^2(Q,\rho)}^2-2t\langle f-A(\bs{\mu}_\lambda)\,,\,A(\bs{\nu})   \rangle
+t^2\|A(\bs{\nu})\|_{L^2(Q,\rho)}^2.
\end{equation}
On the other hand, since it has unit norm $|\bs{\mu}_\lambda|$-a.e.,
the Radon-Nykodim derivative
$\bb{u}_{\bs{\mu}_\lambda}$ has a unique norming functional
 when viewed as an element of 
$(L^1(\bs{\mu}_\lambda))^3$,
given by
\[
\bs{\Psi}\mapsto\int\bs{\Psi}\cdot\bb{u}_{\bs{\mu}_\lambda}\ d|\bs{\mu_\lambda}|,
\qquad \bs{\Psi}\in (L^1(\bs{\mu}_\lambda))^3.
\]
Hence, the $(L^1(\bs{\mu}_\lambda))^3$-norm is G\^ateaux differentiable at
$\bb{u}_{\bs{\mu}_\lambda}$ \cite[Part 3, Ch. 1, Prop. 2, Remark 1]{Beauzamy}
and we get that
\begin{equation}
\label{Gateaux}
\|\bs{\mu}_\lambda+t\bs{\nu}\|_{TV}=\int|\bb{u}_{\bs{\mu}_\lambda}+t\bs{\nu}_a|\,d|\bs{\mu}_\lambda|=\|\bs{\mu}_\lambda\|_{TV}+t
\int\bs{\nu_a}\cdot\bb{u}_{\bs{\mu}_\lambda}d|\bs{\mu}_\lambda| +t\varepsilon(t),
\end{equation}
where $\varepsilon(t)\to0$ when $t\to0$. From (\ref{expq}) and (\ref{Gateaux}),
we gather that
\[
\mathcal{F}_{f,\lambda}(\bs{\mu}_\lambda+t\bs{\nu})-
\mathcal{F}_{f,\lambda}(\bs{\mu}_\lambda)=
-2t\langle A^*(f-A(\bs{\mu}_\lambda))\,,\,\bs{\nu}   \rangle+t\lambda
\int\bs{\nu_a}\cdot\bb{u}_{\bs{\mu}_\lambda}\ d|\bs{\mu_\lambda}|+o(t).
\]
The left hand side is nonnegative by definition of
$\bs{\mu}_\lambda$, so the coefficient of $t$ in the right hand side 
is zero otherwise we could adjust the sign for small $|t|$.
Consequently
\[
\int\left(-2A^*(f-A(\bs{\mu}_\lambda))+\lambda\bb{u}_{\bs{\mu}_\lambda}
\right)\cdot\bs{\nu}_a\,d|\bs{\mu}_\lambda|=0,\qquad
\bs{\nu}_a\in L^1(|\bs{\mu}_\lambda|),
\]
which implies the first equation in (\ref{CP}).

Assume next that the second inequality in (\ref{CP}) is violated:
\begin{equation}
\label{contsi}
|A^*(f-A\bs{\mu}_\lambda)|(x)>\lambda/2
\end{equation} for some $x\in S$.
Then $|\bs{\mu}_\lambda|(\{x\})=0$ by the first part of (\ref{CP})
just proven, and the measure 
\begin{equation}
\label{defdels}
\bs{\nu}=\frac{A^*(f-A\bs{\mu}_\lambda)(x)}{|A^*(f-A\bs{\mu}_\lambda)|(x)}\delta_x,
\end{equation}
is singular with respect to $\bs{\mu}_\lambda$.
Hence, for $t>0$,
\begin{equation}
\label{tvsp}
\|\bs{\mu}_\lambda+t\bs{\nu}\|_{TV}=
\|\bs{\mu}_\lambda\|_{TV}+t\|\bs{\nu}\|_{TV}=
\|\bs{\mu}_\lambda\|_{TV}+t,
\end{equation}
and it follows from (\ref{expq}), (\ref{defdels})
and (\ref{tvsp}) that
\[\mathcal{F}_{f,\lambda}(\bs{\mu}_\lambda+t\bs{\nu})-
\mathcal{F}_{f,\lambda}(\bs{\mu}_\lambda)=-2
t|A^*(f-A\bs{\mu}_\lambda)|(x)+t\lambda +O(t^2)
\]
which is strictly negative for $t>0$ small enough, in view of (\ref{contsi}). 
But this cannot hold since $\bs{\mu}_\lambda$ is a minimizer of 
(\ref{defcrit}), thereby proving the second inequality in (\ref{CP})
by contradiction.

Conversely, assume that (\ref{CP}) holds. Let $\bs{\nu}\in\mathcal{M}(S)^3$ and write the Radon-Nykodim decomposition of $\bs{\nu}$ with respect to $\bs{\mu}_\lambda$ as
$d\bs{\nu}=\bs{\nu}_ad|\bs{\mu}_\lambda|+d\bs{\nu}_s$, where
$\bs{\nu}_a\in L^1(|\bs{\mu}_\lambda|)$ and $\bs{\nu}_s$ is singular 
with respect to $|\bs{\mu}_\lambda|$. Setting $t=1$ in
(\ref{expq}), we get that
\begin{equation}
\label{sufopt}
\begin{array}{rl}\|f-A(\bs{\mu}_\lambda+\bs{\nu})\|_{L^2(Q,\rho)}^2&-\quad
\|f-A(\bs{\mu}_\lambda)\|_{L^2(Q,\rho)}^2\geq-2\langle f-A(\bs{\mu}_\lambda)\,,\,A(\bs{\nu})   \rangle\\
	&=-2\int A^*(f-A(\bs{\mu}_\lambda))\cdot\bs{\nu}_a\, d|\bs{\mu}_\lambda| -
		2\langle A^*(f-A(\bs{\mu}_\lambda))\,,\,\bs{\nu}_s   \rangle\\
	&=-\lambda\int (\bs{\nu}_a\cdot \bb{u}_{\bs{\mu}_\lambda})d|\bs{\mu}_\lambda| -
		2\langle A^*(f-A(\bs{\mu}_\lambda))\,,\,\bs{\nu}_s   \rangle\\
	&\geq-\lambda \int (\bs{\nu}_a\cdot \bb{u}_{\bs{\mu}_\lambda})d|\bs{\mu}_\lambda| -
		\lambda\|\bs\nu_s\|_{TV}.
\end{array}
\end{equation}
In another connection, 
 we have that
\[
\|\bs{\mu}_\lambda+\bs{\nu}\|_{TV}=\int|\bb{u}_{\bs{\mu}_\lambda}+
\bs{\nu}_a|\ d|\bs{\mu}_\lambda|+
\|\bs\nu_s\|_{TV}=
\int(1+2\bs{\nu}_a\cdot\bb{u}_{\bs{\mu}_\lambda}
+|\bs{\nu}_a|^2)^{1/2}\ d|\bs{\mu}_\lambda|
+
\|\bs\nu_s\|_{TV},
\]and since 
$|\bb{u}_{\bs{\mu}_\lambda}|=1$ a.e. with respect to $|\bs{\mu}_\lambda|$,
we obtain:
\begin{equation}
\label{locm}
(1+2\bs{\nu}_a\cdot\bb{u}_{\bs{\mu}_\lambda}
+|\bs{\nu}_a|^2)^{1/2}\geq|1+\bs{\nu}_a\cdot
\bb{u}_{\bs{\mu}_\lambda}
|,\qquad |\bs{\mu}_\lambda|\text{\rm -a.e.}
\end{equation}
Thus, if we let $E_+$ (resp. $E_-$) be the subset of $\text{\rm supp}\,|\bs{\mu}_\lambda|$
where $\bb{u}_{\bs{\mu}_\lambda}\cdot
\bs{\nu}_a>-1$ (resp. $\leq -1$), we obtain:
\begin{equation}
\label{1dec}
\lambda\|\bs{\mu}_\lambda+\bs{\nu}\|_{TV}
\geq \lambda\int_{E_+}(1+\bs{\nu}_a\cdot \bb u_{\bs{\mu}_\lambda}) 
d|\bs{\mu}_\lambda| +
\lambda\|\bs\nu_s\|_{TV},
\end{equation}
Besides, it follows from (\ref{sufopt}) that
\begin{equation}
\label{2dec}
\|f-A(\bs{\mu}_\lambda+\bs{\nu})\|_{L^2(Q,\rho)}^2-
\|f-A(\bs{\mu}_\lambda)\|_{L^2(Q,\rho)}^2\geq
-\lambda \int_{E_+} (\bs{\nu}_a\cdot \bb{u}_{\bs{\mu}_\lambda})d|\bs{\mu}_\lambda| +\lambda \int_{E_-} d|\bs{\mu}_\lambda|
- \lambda\|\bs\nu_s\|_{TV}.
\end{equation}
Adding up (\ref{1dec}) and (\ref{2dec}), using that 
$\|\bs{\mu}_\lambda\|_{TV}=\int_{E_+}d|\bs{\mu}_\lambda|+
\int_{E_-}d|\bs{\mu}_\lambda|$, we obtain:
\begin{equation}
\label{optmeas}
\mathcal{F}_{f,\lambda}(\bs{\mu}_\lambda+\bs{\nu})-
\mathcal{F}_{f,\lambda}(\bs{\mu}_\lambda)\geq0,
\end{equation}
thereby showing that $\bs{\mu}_\lambda$ indeed meets (\ref{crit}).

Finally, observe that in the previous estimates
we neglected the term $t^2\|A\bs{\nu}\|_{L^2(Q,\rho)}^2$ 
in (\ref{expq}) and the term $|\bs{\nu}_a|^2-(\bs{\nu}_a\cdot\bb{u}_{\bs{\mu}_\lambda})^2$ in (\ref{locm}), as well as the term
$\lambda\int_{E_-}(|\bs{\nu}_a\cdot\bb{u}_{\bs{\mu}_\lambda}|-1)d|\bs{\mu}_\lambda|$
in \eqref{1dec}
and \eqref{2dec},
along with the term $\lambda\|\bs{\nu}_s\|_{TV}-2\langle A^*(f-A(\bs{\mu}_\lambda))\,,\,\bs{\nu}_s\rangle$ in (\ref{sufopt}). Hence, equality holds in 
(\ref{optmeas}) if and only if they are all zero.
Thus, for $\bs{\mu}^\prime_\lambda=\bs{\mu}_\lambda+\bs{\nu}$ to be another solution to
(\ref{crit}), it is necessary and sufficient that  $A\bs{\nu}=0$ and
$\bs{\nu}_a=h\bb{u}_{\bs{\mu}_\lambda}$ 
 with $h$ a real-valued function such that $h\geq-1$, a.e. with respect to
 $|\bs{\mu}_\lambda|$, while
$\bs{\nu}_s$ is supported on the subset of $S$ 
where $|A^*(f-A\bs{\mu}_\lambda)|=\lambda/2$
and $A^*(f-A\bs{\mu}_\lambda)=(\lambda/2) \bb{u}_{\bs{\nu}_s}$
at $|\bs{\nu}_s|$-a.e. point.  Thus,
$\bs{\mu}^\prime_\lambda=g\bs{\mu}_\lambda+\bs{\nu}_s$ with $g:=1+h\geq0$,
which gives us (a) and (b).
\end{proof}

That any two minimizers of (\ref{defcrit}) must differ 
by a member of the kernel of $A$ is but a simple consequence of
the strict convexity of the $L^2(Q,\rho)$-norm. In particular,
if the assumptions on $Q,S,\mathcal{A}$ and $v$ of
Lemma \ref{Lem4} hold  and the support of $\rho$ contains $Q\cap\mathcal{A}$,
then any two minimizers are  $S$-equivalent,
by (d) of Proposition \ref{Acomp}. The second assertion of Theorem \ref{VPThm}
means that when  (a) holds, then $\|\bs{\mu}_\lambda'\|_{TV}=\|\bs{\mu}_\lambda\|_{TV}$ if and only if (b) holds.
\begin{corollary}
\label{corspar}
Assumptions and notation being as in Theorem \ref{VPThm}, assume in addition that $S$ is contained in the unbounded connected component of
$\R^3\setminus S$.
Then, the union of the supports of all minimizers of (\ref{defcrit}), for fixed 
$f$ and $\lambda>0$, is contained in a 
finite union of points, embedded curves and surfaces, 
each of which is real-analytic and bounded.
If $\mathcal B$ (resp. $\mathcal{C}$)  is an unbounded connected  real 
analytic surface (resp. curve)
such that $\mathcal{B}\cap Q$ does not disconnect $\mathcal{B}$
(resp. $\mathcal{C}\cap Q=\emptyset$), then the 
aforementioned union of supports has an intersection with $\mathcal{B}$ 
(resp. $\mathcal{C}$)
which is contained in a 
finite union of points and embedded real analytic curves (resp. points).
\end{corollary}
\begin{proof}
Recall from (b) in Proposition~\ref{Acomp} 
that $A^*(f-A\bs{\mu}_\lambda)$ is the restriction to $S$ of a 
$\R^3$-valued real analytic vector field 
on $\R^3\setminus Q$ that vanishes at infinity. Set
$g=|A^*(f-A\bs{\mu}_\lambda)|^2$ which vanishes at infinity and is a real analytic function 
$\R^3\setminus Q\to\R$. 
Theorem \ref{VPThm} implies
that the support of $|\bs{\mu}_\lambda|$, and also of any other minimizer of
\eqref{defcrit}, is included in the zero set of 
$h:=g-\lambda^2/4$. Note that $h$ is independent of the minimizer
$\bs{\mu}_\lambda$ under consideration, since any two have the same image under $A$ by Theorem \ref{VPThm}.
Now, since $h$  is not the zero function on the unbounded component
of $\R^3\setminus Q$ (because $g$ vanishes at infinity), its zero set is a
locally finite union of points and real analytic embedded curves and surfaces,
see discussion after Theorem~\ref{divM}.
Moreover, since $h$ tends to $-\lambda^2/4$  at infinity,
the zero set of $h$ intersected with $S$ is contained in a relatively compact open subset 
of $\R^3\setminus Q$, therefore the support of $\bs{\mu}_\lambda$
or any other minimizer is contained in finitely many of these
points, curves and surfaces.

The proof of the second assertion is similar, reasoning in $\mathcal{B}\setminus Q$ (resp. $\mathcal{C}$ rather than $\R^3\setminus Q$.
\end{proof}
 \begin{rmk}
Corollary \ref{corspar} applies for
instance in paleomagnetism, when
trying to recover  magnetizations on thin slabs of rock 
{\it via} the regularization scheme \eqref{crit}, a case in which 
 $\mathcal{B}$ is a plane. 
Note that if we omit the assumption that $\mathcal{B}$ (resp. $\mathcal{C}$)
is  unbounded
in Corollary \ref{corspar}, then we can only conclude that the
support of $\bs{\mu}_\lambda$ either is contained in finitely many
points and arcs (resp. points) or else $h=0$ on $\mathcal{B}$ (resp. $\mathcal{C}$). 
This remark applies,
{\it e.g.}  in MEG inverse problems
where $\mathcal{B}$ is typically  a closed surface. 
\end{rmk}
Even if $f\in \rnge A$, say $f=A(\bs{\mu}_0)$ for some
$\bs{\mu}_0\in(\mathcal{M}(S))^3$,
it is clear from \eqref{CP} that $\bs{\mu}_\lambda\neq\bs{\mu_0}$
when $\lambda>0$, unless $\bs{\mu}_0=0$. The purpose of the regularizing
term  $\lambda\|\bs{\mu}\|_{TV}$ in \eqref{defcrit} is rather
to get a $\bs{\mu}_\lambda$ which is
not too far from $\bs{\mu}_0$ when $f$ gets replaced by 
$f_e=f+e$ in \eqref{crit}. Here,
$e$ is some error ({\it e.g.} due to measurements)
 and $f_e$ represents the actual data.
To clarify the matter,
whenever $f,e\in L^2(Q,\rho)$ we set $f_e:=f+e$ and,
for $\lambda>0$, we let $\bs{\mu}_{\lambda,e}$ be a minimizer of
\eqref{defcrit} when $f$ gets replaced by $f_e$. Thus, with the notation of
\eqref{crit}, we have that $\bs{\mu}_\lambda=\bs{\mu}_{\lambda,0}$.
Typical results to warrant a regularization approach based on 
approximating $\bs{\mu}_0$ by $\bs{\mu}_{\lambda,e}$
are of  ``consistency'' type,
namely they assert that $\bs{\mu}_{\lambda,e}$ yields 
information on $\bs{\mu}_0$ 
as $\|e\|_{L^2(Q,\rho)}$ and $\lambda$ go to $0$ in a combined fashion, 
see for example \cite[Thms. 2\&5]{BurOsh} or \cite[Thm. 3.5\&4.4]{HoKaPoSch}.
We give below a theorem of this type, which goes beyond
\cite[Thm. 3.5]{HoKaPoSch} in that  we deal not just with weak-$*$ 
convergence of subsequences $\mu_{\lambda_n,e_n}$, but more generally with 
narrow convergence 
of both $\mu_{\lambda_n,e_n}$ and $|\mu_{\lambda_n,e_n}|$.
We will not consider
quantitative convergence properties involving
the Bregman distance, that require an additional source condition which needs 
not be  satisfied here in general.

As an extra piece of notation, we define
for $\bs{\mu}_0\in\mathcal{M}(S)^3$:
\begin{equation}
\label{defmineqA}
\mathfrak{M}(\bs{\mu}_0):=\min \{ \|\bs{\mu}\|_{TV}\colon A(\bs{\mu})=
A(\bs{\mu}_0)\}.
\end{equation}
The infimum in the right-hand side of (\ref{defmineqA})
is indeed attained, by the Banach-Alaoglu theorem and
since the kernel of $A$ is weak-$*$ closed. When $S$ and $Q$ satisfy the 
conditions of Lemma \ref{Lem4}, then this kernel  consists of $S$-silent 
magnetizations and $\mathfrak{M}(\bs{\mu}_0)$ is just $M(\bs{\mu}_0)$ 
defined in \eqref{defmineq}. But when these conditions are not satisfied 
(for instance if $S$ is smooth compact surface), 
then the two quantities may not coincide.

Recall that a sequence $\bs{\mu}_n\in\mathcal{M}(S)^3$ 
converges in the narrow sense to
$\bs{\mu}\in\mathcal{M}(S)^3$ if $\int\varphi\cdot d\bs{\mu_n}\to\int\varphi\cdot d\bs{\mu}$ as $n\to\infty$, whenever $\varphi:S\to\R^3$ is continuous and bounded.
When $S$ is compact this is equivalent to  weak-$*$ convergence, but
if $S$ is unbounded  it means that $\bs{\mu}_n$  does not
``loose mass at infinity''.

\begin{theorem}\label{lamto0Thm}
Assumptions and notation being as in Theorem \ref{VPThm}, 
given $f\in L^2(Q,\rho)$, the following hold.
\begin{enumerate}
\item \label{lamto0a} If  
$f=A( \bs{\mu}_0)$
with $\bs{\mu}_0\in (\mathcal{M}(S))^3$, while 
$e\in L^2(Q,\rho)$ and  $\lambda>0$, then  
\begin{equation}
\label{regbn}
\|\bs{\mu}_{\lambda,e}\|_{TV}\le  \frac{\|e\|_{L^2(Q,\rho)}^2}{\lambda}+
\mathfrak{M}(\bs{\mu}_0)
\end{equation} 
and
\begin{equation}
\label{limTVa}
\lim_{\lambda\to0^+\,,\,\|e\|_{L^2(Q,\rho)}/\sqrt{\lambda}\to0}\|\bs{\mu}_{\lambda,e}\|_{TV}
= \mathfrak{M}(\bs{\mu}_0).
\end{equation} As $\lambda\to 0$ and $\|e\|_{L^2(Q,\rho)}/\sqrt{\lambda}\to0$,
any weak-$*$ cluster point $\bs{\mu}^*$ 
 of $\bs{\mu}_{\lambda,e}$ 
(there must be at least one since  $\|\bs{\mu}_{\lambda,e}\|_{TV}$ is bounded)
meets $A(\bs{\mu}^*)=A(\bs{\mu}_0)=f$ and satisfies:
\begin{equation}\label{minwc}
\|\bs{\mu}^*\|_{TV}= \mathfrak{M}(\bs{\mu}_0).
\end{equation}
Moreover, if 
$\lambda_n\to0^+$ and $\|e_n\|_{L^2(Q,\rho)}/\sqrt{\lambda_n}\to0$, 
with $\lambda_n$, $e_n$ such  that $\bs{\mu}_{\lambda_n,e_n}$ 
converges weak-$*$ to $\bs{\mu}^*$, we have that
\begin{equation}
\label{limfinps}
\lim_{n\to\infty}\int\left|\frac{2A^*(f_{e_n}-A\bs{\mu}_{\lambda_{n},e_n})}{
\lambda_{n}}-\bb{u}_{\bs{\mu}^*}\right|\,d|\bs{\mu}^*|=0,
\end{equation}
also $\bs{\mu}_{\lambda_n,e_n}$ and  $|\bs{\mu}_{\lambda_n,e_n}|$
converge respectively to $\bs{\mu}^*$ and $|\bs{\mu}^*|$ in the narrow sense.
\item \label{lamto0b} If $f\not\in \rnge A$ and
either assumption \mbox{\rm (I)} or \mbox{\rm (II)} 
in Section \ref{MtoFop} holds,
then $\|\bs{\mu}_{\lambda,e}\|_{TV}\to \infty$ as $\lambda\to 0$ and $e\to0$. 
\item If $\lambda\ge 2 \sup_{x\in S} |(A^*f)(x)|$, then the unique minimizer of the right-hand side of \eqref{defmineq} is the zero magnetization. 
\end{enumerate}
\end{theorem}
\begin{proof}

If $f\in\rnge A$, or if  $\rnge A$ is dense in $L^2(Q,\rho)$,
it is clear that $\mathcal{F}_{f_e,\lambda}( \bs{\mu}_{\lambda,e}) \to 0$ as 
$\lambda\to 0$ and $e\to0$,
hence $\| f -A(\bs{\mu}_{\lambda,e})\|_{L^2(Q,\rho)}\to 0$ in this case.
In particular, if $f\not\in \rnge A$ but either assumption (I) or (II) in 
Section \ref{MtoFop} holds, then 
$\rnge A$ is dense by Proposition \ref{Acomp} (c)
and so $\|\bs{\mu}_{\lambda,e}\|_{TV}\to \infty$ 
otherwise a
subsequence would converge weak-$*$ to some  $\bs{\mu}_0\in(\mathcal{M}(S))^3$ 
implying in the limit that $f=A(\bs{\mu}_0)$, a contradiction
which proves \ref{lamto0b}.

Next, let $\widetilde{\bs{\mu}}_0$ be a minimizer of the right hand side of
\eqref{defmineq}, so that $A(\widetilde{\bs{\mu}}_0)=f$ and
$\|\widetilde{\bs{\mu}}_0\|_{TV}=\mathfrak{M}(\bs\mu_0)$.
By the optimality of $\bs{\mu}_{\lambda,e}$,
we have  that
\begin{equation}
\label{ptdep}
\begin{array}{l}
\|f_e-A(\bs{\mu}_{\lambda,e})\|^2_{L^2(Q)}+\lambda\| \bs{\mu}_{\lambda,e}\|_{TV} = 
\mathcal{F}_{f_e,\lambda}( \bs{\mu}_{\lambda,e})\le \mathcal{F}_{f_e,\lambda}( \widetilde{\bs{\mu}}_0)\\
= \|e\|_{L^2(Q,\rho)}^2+ \lambda\| \widetilde{\bs{\mu}}_0\|_{TV}=
\|e\|_{L^2(Q,\rho)}^2+\lambda \mathfrak{M}(\bs{\mu}_0),
\end{array}
\end{equation}
implying that \eqref{regbn} holds.
Thus, if $\bs{\mu}^*$ is a weak-$*$ cluster point of $\{\bs{\mu}_{\lambda,e}\}$
as $\lambda\to0^+$ with $\|e\|_{L^2(Q,\rho)}=o(\sqrt{\lambda})$,
and if $\lambda_n$, $e_n$  are sequences with these limiting properties
such  that $\bs{\mu}_{\lambda_n,e_n}$ 
converges weak-$*$ to $\bs{\mu}^*$,
we deduce  from the Banach-Alaoglu Theorem that
\begin{equation}
\label{premlim}
\|\bs{\mu}^*\|_{TV}\le \liminf_n\|\bs{\mu}_{\lambda,e}\|_{TV}\le
\limsup_n\|\bs{\mu}_{\lambda,e}\|_{TV}\le \mathfrak{M}(\bs{\mu}_0).
\end{equation}
Also, since $A$ is weak-$*$ to weak continuous (it is even compact), 
we get that
\begin{equation}
\label{seclim}
\|f-A(\bs{\mu}^*)\|_{L^2(Q,\rho)}\le
\lim_n\|f-A(\bs{\mu}_{\lambda,e})\|_{L^2(Q,\rho)}=0,
\end{equation}
where the last equality was obtained in the proof of (b).
From \eqref{seclim} it follows that  $ A(\bs{\mu}^*)=f$,
and from \eqref{premlim} we now see that \eqref{minwc}
holds, by definition of $\mathfrak{M}(\bs{\mu}_0)$.
Moreover, since a weak-$*$ convergent subsequence can be extracted from any 
subsequence of $\bs{\mu}_{\lambda,e}$, we deduce from what precedes
that \eqref{limTVa} takes place. 
Next, if $\lambda_n$, $e_n$ are as before, we get 
in
view of \eqref{limTVa} and (\ref{ptdep}) that 
$\|f_{e_n}-A(\bs{\mu}_{\lambda_n,e_n})\|^2_{L^2(Q,\rho)}=
o(\lambda_n)$,  
which is equivalent to
\[
0=\lim_n\frac{2}{\lambda_n}\langle f_{e_n}-A(\bs{\mu}_{\lambda_n,e_n}),
f-A(\bs{\mu}_{\lambda_n,e_n})\rangle+
\frac{2}{\lambda_n}\langle  f_{e_n}-A(\bs{\mu}_{\lambda_n,e_n}),
e_n\rangle,
\]
and since $\|e_n\|_{L^2(Q,\rho)}=o(\sqrt{\lambda_n})$ while $f=A(\bs{\mu}^*)$,
we obtain:
\begin{equation}
\label{limsc}
\begin{array}{l}
0=
\lim_n\left\langle \frac{2A^*(f_{e_n}-A(\bs{\mu}_{\lambda_n,e_n}))}{\lambda_n}\,,\,
\bs{\mu}_{\lambda_n,e_n}-\bs{\mu}^*\right\rangle=
\lim_n\left\langle \frac{2A^*(f_{e_n}-A(\bs{\mu}_{\lambda_n,e_n}))}{\lambda_n}\,,\,
\bs{\mu}_{\lambda_n,e_n}\right\rangle\\
-\lim_n \left\langle
\frac{2A^*(f_{e_n}-A(\bs{\mu}_{\lambda_n,e_n}))}{\lambda_n}\,,\,
\bs{\mu}^*\right\rangle
=\lim_n\left( \|\bs{\mu}_{\lambda_n,e_n}\|_{TV}-
\int \frac{2A^*(f_{e_n}-A(\bs{\mu}_{\lambda_n,e_n}))}{\lambda_n} \cdot d\bs{\mu}^*\right)\\
=
\|\bs{\mu}^*\|_{TV}-\ \lim_n\int \frac{2A^*(f_{e_n}-A(\bs{\mu}_{\lambda_n,e_n}))}{\lambda_n} \cdot\bb{u}_{\bs{\mu}^*}d|\bs{\mu}^*|,
\end{array}
\end{equation}
where we used the first relation in (\ref{CP}) to get the third equality
and \eqref{limTVa}, \eqref{minwc} to get the last one.
By the  second relation in (\ref{CP}), we know  that
$|2A^*(f_{e_n}-A\bs{\mu}_{\lambda_n})/\lambda_n| \leq1$ everywhere on $S$,
hence \eqref{limsc} implies that for any $\varepsilon>0$
\[
\limsup_n|\bs{\mu}^*|\left\{x\in S:\ \frac{2A^*(f_{e_n}-A\bs{\mu}_{\lambda_n,e_n})}
{\lambda_n}\cdot
\bb{u}_{\bs{\mu}^*}<1-\varepsilon\right\}=0.
\]
Therefore, using  the Borel-Cantelli Lemma and a diagonal argument, we
may extract a subsequence $\lambda_{k_n}$ for which
$2A^*(f_{e_{k_n}}-A\bs{\mu}_{\lambda_{k_n},e_{k_n}})/
\lambda_{k_n}$ converges pointwise $|\bs{\mu}^*|$-a.e. to 
$\bb{u}_{\bs{\mu}^*}$. So, by dominated convergence, it follows that
\[
\lim_{n\to\infty}\int\left|\frac{2A^*(f_{e_{k_n}}-A(\bs{\mu}_{\lambda_{k_n},e_{k_n}}))}{
\lambda_{k_n}}-\bb{u}_{\bs{\mu}^*}\right|\,d|\bs{\mu}^*|=0,
\]
and since the reasoning can be applied to any subsequence of $\lambda_n$,
we obtain \eqref{limfinps}. 

We now prove that $|\bs{\mu}_{\lambda_n,e_n}|$ converges weak-$*$ to 
$|\bs{\mu}^*|$. For this, it is  enough to show that if
$|\bs{\mu}_{\lambda_n,e_n}|$ converges weak-$*$ to
$\bs{\nu}\geq0\in\mathcal{M}(S)$, then $\bs{\nu}=|\bs{\mu}^*|$.
For this,
let $\psi:S\to[0,1]$ be a continuous function with 
compact support.
Pick $\varepsilon>0$, and then $n_\varepsilon$ such that the 
integral in the left-hand 
side  of \eqref{limfinps} is less than $\varepsilon$ for $n\geq n_\varepsilon$.
 As $|2A^*(f_{e_{n_\varepsilon}}-A(\bs{\mu}_{\lambda_{n_\varepsilon},e_{n_\varepsilon}}))/\lambda_{n_\varepsilon}| \leq1$ everywhere on $S$ by  (\ref{CP}), we obtain from the definition of $n_\varepsilon$ and \eqref{minwc} that
\[
\begin{array}{l}
\int\psi\,d\bs{\nu}=\lim_n\int\psi\,d|\bs{\mu}_{\lambda_n,e_n}|
\geq
\lim_n\left|\int\psi\frac{2A^*(f_{e_{n_\varepsilon}}-A(\bs{\mu}_{\lambda_{n_\varepsilon},e_{n_\varepsilon}}))}{\lambda_{n_\varepsilon}}\cdot d\bs{\mu}_{\lambda_n,e_n}\right|\\
=
\left|\int\psi\frac{2A^*(f_{e_{n_\varepsilon}}-A(\bs{\mu}_{\lambda_{n_\varepsilon},e_{n_\varepsilon}}))}{\lambda_{n_\varepsilon}}\cdot 
d\bs{\mu}^*\right|\geq 
\int\psi d|\bs{\mu}^*|-
\int\left|\frac{2A^*(f_{e_{n_\varepsilon}}-A(\bs{\mu}_{\lambda_{n_\varepsilon},e_{n_\varepsilon}}))}{\lambda_{n_\varepsilon}}-\bf{u}_{\bs{\mu}^*}\right|\cdot d|\bs{\mu}^*|\\
\geq \int\psi d|\bs{\mu}^*|-\varepsilon,
\end{array}
\]
where we used in the equality above that $A^*(f_{e_{n_\varepsilon}}-A(\bs{\mu}_{\lambda_{n_\varepsilon},e_{n_\varepsilon}}))$ is continuous on $S$, by 
Proposition \ref{Acomp} (b).
Since $\varepsilon>0$ was arbitrary, we conclude that $\bs{\nu}-|\bs{\mu}^*|\geq0$. However, since $\||\bs{\mu}^*|\|_{TV}=\mathfrak{M}(\bs{\mu}_0)$ by
\eqref{minwc}, whereas $\|\bs{\nu}\|_{TV}\leq\mathfrak{M}(\bs{\mu}_0)$ by the Banach Alaoglu theorem, we conclude that  $\bs{\nu}-|\bs{\mu}^*|$ is the zero measure, as desired.

To establish  that $\bs{\mu}_{\lambda_n,e_n}$ converges to $\bs{\mu}^*$ 
in the narrow sense, pick $\varepsilon>0$ and $n_\varepsilon$ as before.
Fix $R_\varepsilon$ so large that $|\bs{\mu}^*|(S\cap \overline{B}(0,R_\varepsilon)>\|\bs{\mu}^*\|_{TV}-\varepsilon$ and
for each $R$ let $\psi_R:S\to[0,1]$ be continuous, identically 1 on
$S\cap B(0,R)$ and 0 on $S\setminus B(0,2R)$. 
Reasoning as before, we get that
\[
\begin{array}{l}
\|\bs{\mu}^*\|_{TV}\geq \limsup_n\,|\bs{\mu}_{\lambda_n,e_n}|(S\cap \overline{B}(0,2R_\varepsilon))\geq\limsup_n\,
\int\psi_{R_\varepsilon}\,d|\bs{\mu}_{\lambda_n,e_n}|\\
\geq
\lim_n\left|\int\psi_{R_\varepsilon}\frac{2A^*(f_{e_{n_\varepsilon}}-A(\bs{\mu}_{\lambda_{n_\varepsilon},e_{n_\varepsilon}}))}{\lambda_{n_\varepsilon}}\cdot d\bs{\mu}_{\lambda_n,e_n}\right|=
\left|\int\psi_{R_\varepsilon}\frac{2A^*(f_{e_{n_\varepsilon}}-A(\bs{\mu}_{\lambda_{n_\varepsilon},e_{n_\varepsilon}}))}{\lambda_{n_\varepsilon}}.
d\bs{\mu}^*\right|
\\
\geq
\left|\int\psi_{R_\varepsilon} \bf{u}_{\bs{\mu}^*}\cdot d\bs{\mu}^*\right|-
\left|\int\psi_{R_\varepsilon} \left(\frac{2A^*(f_{e_{n_\varepsilon}}-A(\bs{\mu}_{\lambda_{n_\varepsilon},e_{n_\varepsilon}}))}{\lambda_{n_\varepsilon}}-\bf{u}_{\bs{\mu}^*}\right)\cdot d\bs{\mu}^*\right|\\
\geq \int\psi_{R_\varepsilon}d|\bs{\mu}^*|-
\int\left|\frac{2A^*(f_{e_{n_\varepsilon}}-A(\bs{\mu}_{\lambda_{n_\varepsilon},e_{n_\varepsilon}}))}{\lambda_{n_\varepsilon}}-\bf{u}_{\bs{\mu}^*}\right|\cdot d|\bs{\mu}^*|
\geq
|\bs{\mu}^*|(S\cap \overline{B}(0,R))-\varepsilon\\
\geq \|\bs{\mu}^*\|_{TV}-2\varepsilon.
\end{array}
\]
Hence, in view of \eqref{limTVa} and \eqref{minwc},
we see from what precedes
that for $n$ large enough
$|\bs{\mu}_{\lambda_n,e_n}|(S\setminus \overline{B}(0,2R_\varepsilon))\leq3\varepsilon$,
say.
Therefore, if we fix
a bounded and continuous $\varphi: S\to\R^3$ with $|\varphi|\leq M$, 
we have since $\bs{\mu}_{\lambda_n,e_n}$ converges weak-$*$ to $\bs{\mu}^*$
that
\[
\begin{array}{l}\limsup_n\left|\int\varphi\cdot d(\bs{\mu}_{\lambda_n,e_n}-\bs{\mu}^*)\right|\leq
\limsup_n\left|\int\psi_{2R_\varepsilon}\varphi\cdot d(\bs{\mu}_{\lambda_n,e_n}-\bs{\mu}^*)\right|\\
+\limsup_n \left|
\int(1-\psi_{2R_\varepsilon})\varphi\cdot d(\bs{\mu}_{\lambda_n,e_n}-\bs{\mu}^*)\right|\leq0+
6M\varepsilon.
\end{array}\]
Because $\varepsilon>0$ was arbitrary, we deduce that 
$\bs{\mu}_{\lambda_n,e_n}$ converges to $\bs{\mu}^*$ in the narrow sense,
and the fact that 
$|\bs{\mu}_{\lambda_n,e_n}|$ converges to $|\bs{\mu}^*|$ in the narrow sense
as well can be shown in a similar way.
This proves (a).


Finally, suppose $\lambda\ge 2 \sup_{x\in S} |(A^*f)(x)|$.  Theorem~\ref{VPThm} shows that the zero magnetization is a minimizer of $\mathcal{F}_{f,\lambda}$ and that any other minimizer $\bs{\mu}$ must be silent, but then 
$\mathcal{F}_{f,\lambda}(\bs{\mu} )=\|f\|+\lambda \|\bs{\mu}\|_{TV}$ showing that in fact the zero magnetization is the unique minimizer of $\mathcal{F}_{f,\lambda}$.  
\end{proof}

Assertion (a) of Theorem \ref{lamto0Thm} entails that any sequence
$\mu_{\lambda_n,e_n}$ with
$\lambda_n=o(1)$ and $\|e_n\|_{L^2(Q)}=o(\sqrt{\lambda_n})$
has a subsequence converging in the narrow sense to some 
$\bs{\mu}^*$ such that $A(\bs{\mu}^*)=A(\bs{\mu}_0)=f$ and 
$\|\bs{\mu}^*\|_{TV}=\mathfrak{M}(\bs{\mu}_0)$. If such a $\bs{\mu}^*$ is 
unique,  we get narrow convergence  of
$\bs{\mu}_{\lambda,e}$ to $\bs{\mu}^*$
as soon as $\lambda\to0$ and $\|e\|_{L^2(Q)}/\sqrt{\lambda}\to0$.
Using Theorems \ref{unrec3D}  and  \ref{UnidirThm}, we thus obtain
a recovery result for ``sparse magnetizations'' as follows.
\begin{theorem}\label{weak*to0}
	Let $S,Q\subset\R^3$ satisfy the assumptions of Lemma \ref{Lem4}
with $Q$ compact, and assume in addition that $S$ is a
 slender set  with $S=\bigcup_{i=1}^n S_i$ for some finite 
collection of disjoint closed sets $S_1, S_2, \ldots, S_n$.   
	Suppose $\bs{\mu}_0\in\mathcal{M}(S)^3$ and set $f=A\bs \mu_0$.
	If either
	\begin{enumerate}
		\item $\bs{\mu}_0\mathcal{b}S_i$ is uni-directional for $i=1, 2, \ldots, n$,
		\item or $\text{\emph{supp}}\ \bs\mu_0$ is purely 1-unrectifiable
		\end{enumerate}
		then $\bs{\mu}_{\lambda,e}$ converges narrowly
to $\bs{\mu}_0$ and $|\bs{\mu}_{\lambda,e}|$ converges narrowly
to $|\bs{\mu}_0|$as $\lambda\to0$ and $\|e\|_{L^2(Q)}/\sqrt{\lambda}\to0$.
		\end{theorem}

\begin{rmk}
In the setting of Theorem \ref{lamto0Thm} (a),
it is generally not true that 
$\|\bs{\mu}^*-\bs{\mu}_{\lambda_n,e_n}\|_{TV}\to0$.
For instance, 
let $S\subset\R^2\times\{0\}$ be compact, assume that $Q\subset\R^2\times\{h\}$ for some $h>0$, let $\rho$ be 2-dimensional Hausdorff measure and 
$\bs\mu_0=\chi_Sv$, where $v\in\R^3$.
Then $\bs{\mu}_0$ is unidirectional,  and we 
know from Theorem \ref{weak*to0} 
that $\bs{\mu}_{\lambda,e}$ converges narrowly
to $\bs{\mu}_0$ as $\lambda\to0$ and $\|e\|_{L^2(Q)}/\sqrt{\lambda}\to0$.
Still, the
support of $\bs{\mu}_{\lambda,e}$ has 
Hausdorff dimension at most 1, by Corollary \ref{corspar}, therefore
$\bs\mu_{\lambda,e}$ and $\bs{\mu}_0$ are mutually singular.
Hence $\|\bs{\mu}_{\lambda,e}-\bs{\mu}_0\|_{TV}=\|\bs{\mu}_{\lambda,e}\|_{TV}
+\|\bs{\mu}_0\|_{TV}$ cannot go to zero when $\lambda$ goes to zero.
\end{rmk}
As a particular case of Theorem \ref{weak*to0}, taking into account 
Corollary \ref{corspar}, we obtain:
\begin{corollary}
\label{points}
Let $S,Q\subset\R^3$ satisfy the assumptions of Lemma \ref{Lem4}
with $Q$ compact and $S$  slender.
Supppose that $\bs{\mu}_0\in\mathcal{M}(S)^3$ is a finite sum of point dipoles:
 $\bs{\mu}_0=\sigma_{j=1}^Nv_j\delta_{x^{(j)}}$ for some
$v_j\in\R^3\setminus\{0\}$ and $x^{(j)}\in S$, and set $f=A\bs \mu_0$. 
Then, to each $\varepsilon,r>0$ there is $\eta>0$ such that,
whenever $0<\lambda<\eta$ and $\|e\|_{L^2(Q)}/\sqrt{\lambda}<\eta$:
\begin{enumerate}
\item $\bigl|\,|\bs{\mu}_{\lambda,e}|(B(x^{(j)},r))-|v_j|\,\bigr|<\varepsilon$ for $1\leq j\leq N$,
\item $|\bs{\mu}_{\lambda,e}|\left(S\setminus\cup_{j=1}^NB(x^{(j)},r)\right)<\varepsilon$,
\item $\bigl|\,\bs{\mu}_{\lambda,e}(B(x^{(j)},r))-v_j\,\bigr|<\varepsilon$ for $1\leq j\leq N$,
\item $\text{\emph{supp}}\ \bs\mu_{\lambda,e}$ is contained in a finite collection of analytic surfaces, curves and points.
\end{enumerate}
\end{corollary}

\begin{rmk}
If $S$ is contained in some unbounded 
real analytic surface or curve ({\it e.g.} a plane or a line) which is
disjoint from 
$Q$, point (d) of Corollary \ref{points} involves only curves and points
in case of a surface and only points in case of a curve. This follows from 
Corollary \ref{corspar}.
\end{rmk}


\begin{figure}[h!]
	\includegraphics[scale=.5]{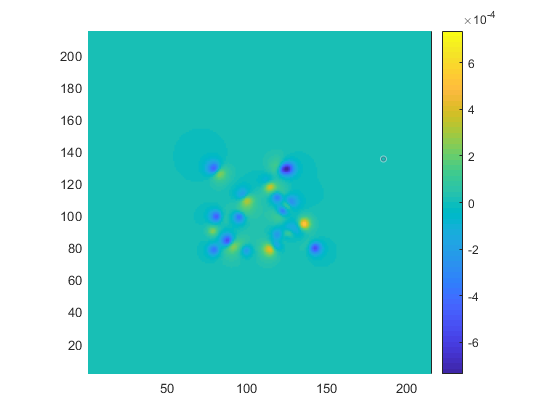} \includegraphics[scale=.5]{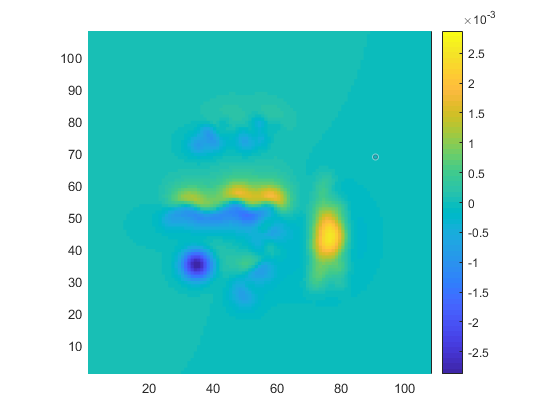}
	\label{fig:Fields}

	\caption{The fields $b_3(\bs{\mu})$ for (left) a sparse magnetization consisting of 20 dipoles (see Figure~\ref{fig:Sparse} for $\bs\mu$ and reconstructions $\bs{\mu}_\lambda$) and (right) a piecewise unidirectional magnetization with $S$ consisting of four connected components (see Figure~\ref{fig:Uni}).}
\end{figure}

\begin{figure}[h!]
			\includegraphics[scale=.5]{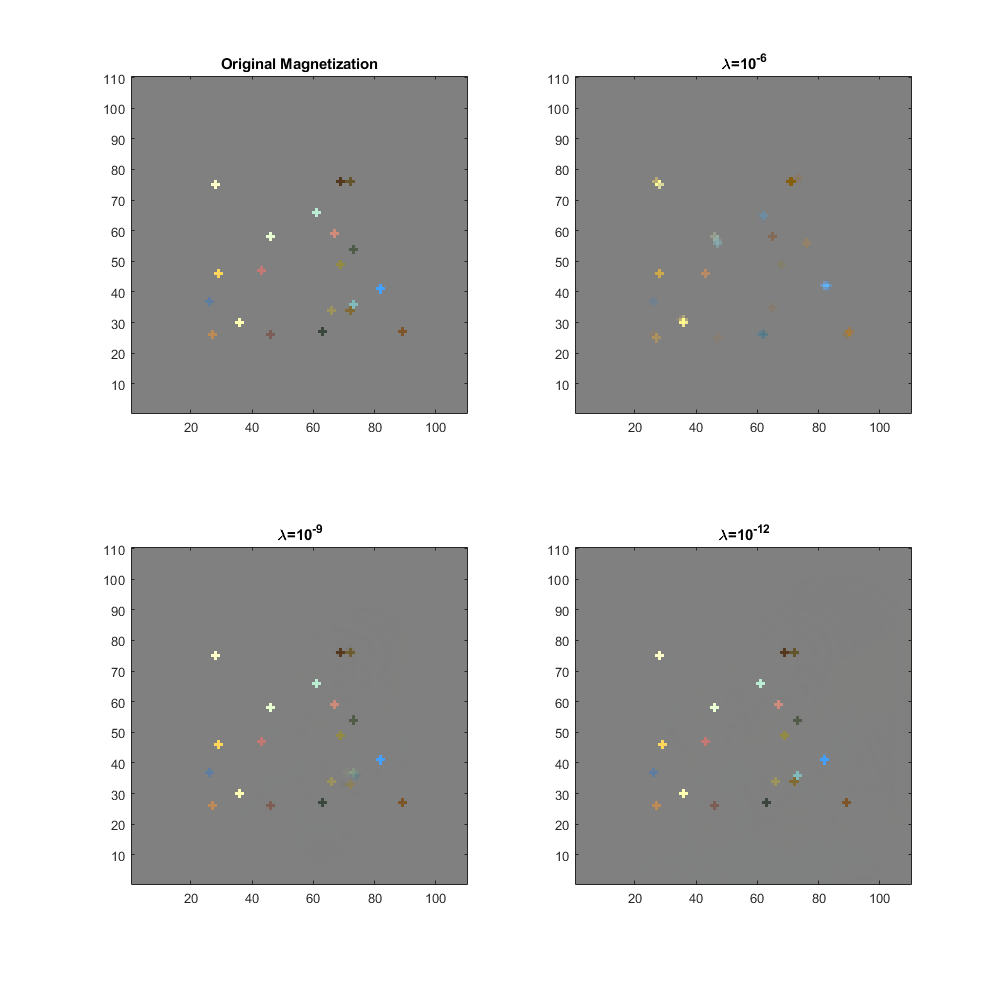}
	\caption{Example for a Sparse Magnetization. 
		Each figure shows for its respective magnetization the magnitude and direction of its dipoles by representing different vectors with diferent colors. The closer a dipole is to zero the closer its assigned color is to grey. 
		The relative distances in total variation to $\bs{\mu}_0$ are 1.294, 0.207 and 0.004 respectively, confirming convergence as expected.
		(Convolution with a 3 by 3 cross matrix was used to increase visibility)}
		\label{fig:Sparse}
\end{figure}

\begin{figure}[hpb]
	\includegraphics[scale=.55]{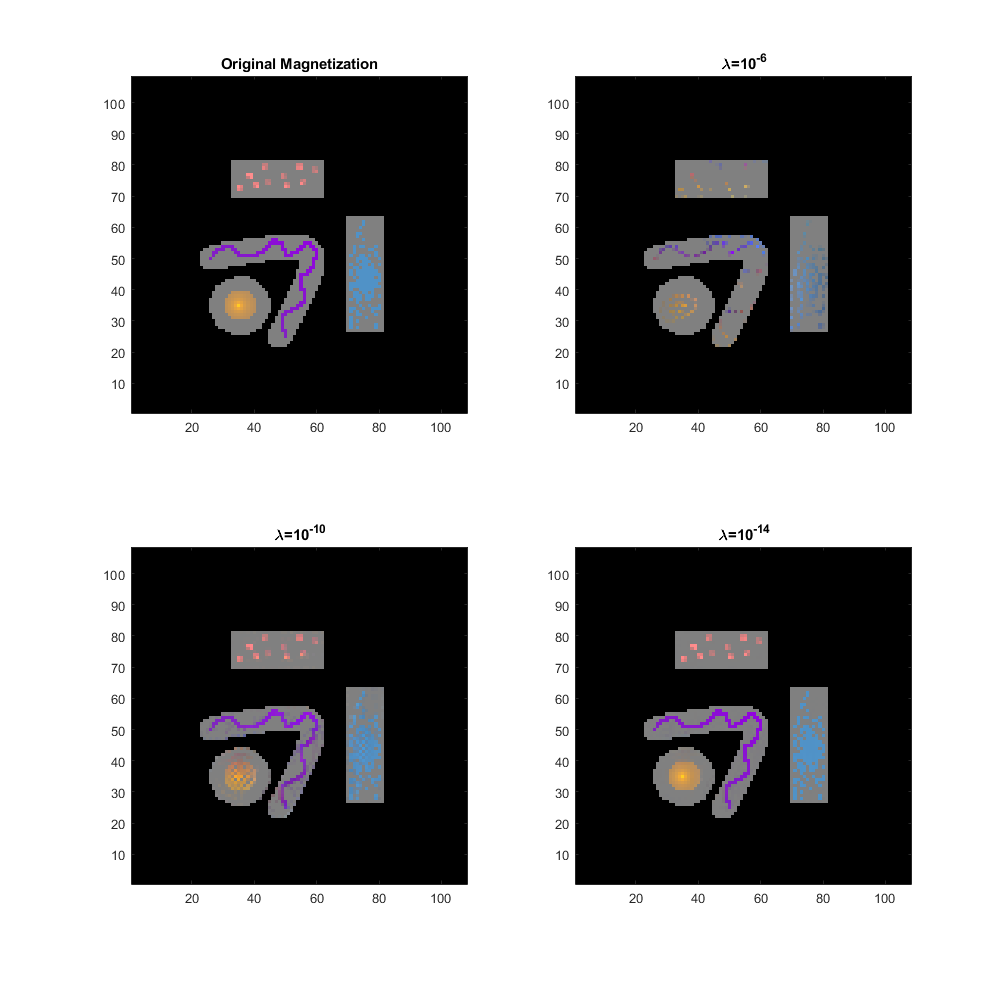}
	\caption{Example for a piece-wise Unidirectional Magnetization.
		The figures are made similarly to the ones of Figure~\ref{fig:Sparse} with black representing the complement of $S$.
		Here the relative distances are 1.167, 0.247 and 0.015}
	\label{fig:Uni}
\end{figure}

\section{Numerical Examples}
\label{Numerical}

In this section we present two examples of numerical reconstructions  illustrating Theorem~\ref{weak*to0}. 
In both cases, we are considering continuous problems with $v=e_3$ (the measured field is $b_3(\bs{\mu}_0)$), $e=0$ (no noise),
and   $S$ and $Q$    compact subsets of the $z=0$ and $z=h$ planes, respectively.  
We discretize the continuous problems by restricting to magnetizations $\bs{\mu}$'s  consisting of a finite number of  dipoles   located on a rectangular grid in the $z=0$ plane intersected with $S$ and   samples    $b_3(\bs{\mu})$ evaluated at points in a rectangular grid in the parallel $z=h$ plane intersected with a rectangle $Q$.   Numerical solutions to the discretized problems are then obtained using the Fast Iterative Shrinkage-Thresholding algorithm (FISTA) from \cite{FISTA} together with the In-Crowd algorithm from \cite{InCrowd}.  
We will consider in  more detail in forthcoming work the relations between the solutions of such  discretized problems and the associated continuous problems of the type addressed in this paper and we will also address the algorithmic and computational details for obtaining solutions to these discrete problems.   The examples provided here are only intended for illustrative purposes.   



The continuous problem for the first example is designed to illustrate the recovery of magnetizations with sparse support as in part (b) of Theorem~\ref{weak*to0}. In this example $S$ and $Q$ are squares in the planes   $z=0$ and $z=h=.1$, respectively, and $\bs{\mu}_0$ consists of  20 dipoles  in $S$ with  moments of differing directions.  The source and measurement  grids both have  $.0187\times .0187$ spacing.  The source grid   consists of $108 \times 108$ points and the moments of each dipole are allowed to take on any values in $\R^3$.    The measurement grid consists of $215\times 215$ points.   Reconstructions are computed for three values of $\lambda$: $10^{-6}$, $10^{-9}$ and $10^{-12}$.    
 
 In the second case, $S$ consists of the union of four disjoint compact regions as shown in Figure~\ref{fig:Uni}.  The restriction of the magnetization $\bs{\mu}_0$ to each component is unidirectional as in part (a) of Theorem~\ref{weak*to0}.   The source grid now consists of the grid points that are contained in the set $S$ and again the moments of these dipoles are unconstrained; i.e., there is no uni-directionality assumption when solving for a reconstruction.   
 The $\bs{\mu}_\lambda$ are taken among all magnetizations supported in those regions for $\lambda$ equal to $10^{-6}$, $10^{-10}$ and $10^{-14}$.

 \bibliographystyle{abbrv}

\bibliography{SparsityTVrefs}

\appendix

\section{Jordan-Brouwer separation theorem
in the non compact case}
\label{app1}
In this section, we record a proof of  the
 Jordan-Brouwer separation theorem for smooth and connected,
complete but not necessarily compact surfaces in $\R^3$. The argument applies 
in any dimension. We are confident the result is known,
but we could not find a
published reference. More general proofs, valid 
for non-smooth manifolds as well,
could be given using deeper facts from algebraic toplogy.
For instance, one based on  Alexander-Lefschetz duality
can be modelled after Theorem 14.13 in
www.seas.upenn.edu/~jean/sheaves-cohomology.pdf  (which deals with compact
topological  manifolds). More precisely, using Alexander-Spanier cohomology with compact 
support  and appealing to
\cite[ch. 6, sec. 6 cor. 12 and sec. 9 thm. 10]{Spanier}, 
one can generalize the proof just mentioned to  handle the case
of non-compact manifolds. Herafter, we merely deal with smooth surfaces
and rely on basic
notions from differential-topology, namely 
intersection theory  modulo 2.

Recall that a smooth  manifold $X$ of dimension $k$ embedded in $\R^n$ is a subset of 
the latter, 
each point of which has a neighborhood $V$ such that $V\cap X=\phi(U)$
where $U$ is an open subset of $\R^k$ and $\phi:U\to\R^n$ a $C^\infty$-smooth 
injective map with injective  derivative at every point. 
The map $\phi$ is called a 
parametrization of $V$ with domain $U$, and the image of its derivative
$D\phi(u)$ at $u$ is the tangent space to $X$ at $\phi(u)$, hereafter 
denoted by
$T_{\phi(u)}X$.
Then, by the constant rank theorem, there is an open set
$W\subset\R^n$ with $W\cap X=V$ and a $C^\infty$-smooth map $\psi:W\to U$
such that $\psi\circ \phi=\textrm{id}$, the identity map of $U$. The 
restriction $\psi_{|V}$ is called a chart with domain $V$. This allows 
one to carry over to $X$ 
local tools from differential calculus, see \cite[ch. 1]{GP}.
We say that $X$ is closed if it is a closed subset of $\R^n$.

If $X$, $Y$ are smooth manifolds embedded in $\R^m$ and $\R^n$ respectively,
and if $Z\subset Y$ 
is a smooth embedded submanifold, a smooth map
$f:X\to Y$ is said to be transversal to $Z$ if
$\textrm{Im} Df(x)+T_{f(x)}Z=T_{f(x)}Y$ at every $x\in X$ 
such that $f(x)\in Z$.
If $f$ is transversal to $Z$, then $f^{-1}(Z)$ is 
an embedded submanifold of $X$ whose codimension is the same as the 
codimension of  $Z$ in $Y$.  In particular, if $X$ is compact and
$\textrm{dim} X+\textrm{dim} Z=
\textrm{dim} Y$, then $f^{-1}(Z)$ consists of finitely many points.
The residue class modulo 2 of the cardinality of such points is 
the intersection number of $f$ with $Z$ modulo 2, denoted by $I_2(f,Z)$.
If in addition $Z$ is closed in $Y$, then $I_2(f,Z)$ is invariant under 
small homotopic deformations of $f$, and this allows one to define
$I_2(f,Z)$ even when  $f$ is not transversal to $Z$, because a suitable
but arbitrary small homotopic deformation of $f$ will guarantee transversality,
see \cite[ch. 2]{GP}.
\begin{theorem} 
\label{JB}If  $\mathcal{A}$ is a $C^\infty$-smooth complete and  connected
surface embedded in $\R^3$, then $\R^3\setminus\mathcal{A}$ has two connected 
components. 
\end{theorem}
\begin{proof}
Let $W$ be a tubular neighborhoud of $\mathcal{A}$ in $\R^3$
\cite[Ch. 2, Sec. 3, ex. 3 \& 16]{GP}. That is, 
$W$ is an open  neighborhood of $\mathcal{A}$ in $\R^3$ comprised of points
$y$ having a unique closest point from $X$, say $x$, such that
$|y-x|<\varepsilon(x)$ where $\varepsilon$ is a suitable
smooth and strictly positive function on $\mathcal{A}$.
Thus, we can write $W=\{x+tn(x),\,x\in\mathcal{A},\,|t|<\varepsilon(x)\}$, 
where $n(x)$ is a  
normal vector to $\mathcal{A}$  at $x$ of unit length.
Note that, for each $x\in \mathcal{A}$, 
there are two possible (opposite) 
choices of $n(x)$, but the definition of $W$ makes it irrelevant which one 
we make. Moreover, if we fix $n(x)$ and $\eta\in(0,1)$,
we can find a neighborhood $V$  of $x$ in $\mathcal{A}$  such that, to each 
$y\in V$, there is a unique choice of $n(y)$ with
$|n(y)-n(x)|<\eta$. Indeed, if $\phi:U\to V$ 
is a parametrization with inverse $\psi$ such that
$x\in V$, and if we set
$n_y:=\partial_{x_1}\phi(\psi(y))\wedge\partial_{x_2}\phi(\psi(y))/\|\partial_{x_1}\phi(\psi(y))\wedge\partial_{x_2}\phi(\psi(y))\|$
where $x_1$, $x_2$ are Euclidean coordinates on $U\subset\R^2$
while  $\partial_{x_j}$ 
denotes  the partial derivative with respect to $x_j$ and
the wedge indicates the vector product,
then the two possible choices for $n(y)$ when $y\in V$ are $\pm n_y$.
Thus, if
we select for instance $n(x)=n_x$ and subsequently set $n(y)=n_y$, 
we get upon shrinking $V$ if necessary that
$|n(x)-n(y)|<\eta$ and $|n(x)+n(y)|>2-\eta$ for $y\in V$.
As a consequence, if $\Upsilon:[a,b]\to \mathcal{A}$ is a continuous path,
and if $n_b$ is a unit normal vector to $\mathcal{A}$ at $\Upsilon(b)$, there 
is a continuous choice of $n(\Upsilon(\tau))$ for $\tau\in[a,b]$ such that
$n(\Upsilon(b))=n_b$.

Fix $x_0\in \mathcal{A}$ and let $n_0$ be an arbitrary choice for $n(x_0)$.
Pick $t_0$ with
$0<t_0<\varepsilon(x_0)$,  and define two points in $W$
by $x_0^\pm=x_0\pm t_0 n_0$.
\emph{We claim} that each $y\in \R^3\setminus \mathcal{A}$ can be joined either to 
$x_0^+$ or to $x_0^-$ by a continuous arc contained in $\R^3\setminus \mathcal{A}$. 
Indeed, let $\gamma:[0,1]\to \R^3$ be a continuous path with
$\gamma(0)=y$ and  $\gamma(1)=x_0$. Let $\tau_0\in(0,1]$ be smallest such that
$\gamma(\tau_0)\in \mathcal{A}$; such a $\tau_0$ exists since $\mathcal{A}$ is closed. 
Pick $0<\tau_1<\tau_0$ close enough to $\tau_0$ that 
$\gamma(\tau_1)\in W$, say
$\gamma(\tau_1)=x_1+t_1 n_1$ where $x_1\in \mathcal{A}$, $|t_1|<\varepsilon(x_1)$,
and $n_1$ is a unit vector normal to $\mathcal{A}$ at $x_1$.
Since $\mathcal{A}$ is connected, there is a continuous path
$\Upsilon:[\tau_1,1]\to \mathcal{A}$ such that $\Upsilon(\tau_1)=x_1$ and
$\Upsilon(1)=x_0$. Along the path $\Upsilon$, there is a continuous
choice of $\tau\to n(\Upsilon(\tau))$ such that $n(x_0)=n_0$; this follows 
from a previous remark.  Changing the sign of $t_1$
if necessary, we may assume that $n_1=n(x_1)$. Let $\eta:[\tau_1,1]\to\R_+$
be a continuous function such that $0<|\eta(\tau)|<\varepsilon(\Upsilon(\tau))$
with $\eta(\tau_1)=t_1$ and $\eta(1)=\textrm{sgn}\,t_1|t_0|$. Such an $\eta$ 
exists, since $\varepsilon$ is continuous and strictly 
positive while $|t_1|<\varepsilon(x_1)$ and
$|t_0|<\varepsilon(x_0)$. Now the concatenation of $\gamma$ restricted
to $[0,\tau_1]$ and  $\gamma_1:[\tau_1,1]\to \R^3$
given by $\gamma_1(\tau)=\Upsilon(\tau)+\eta(\tau)n(\Upsilon(\tau))$ is 
a continuous path from $y$ to either $x_0^+$ or $x_0^-$ (depending on the
sign of $t_1$) which is entirely contained in $\R^3\setminus\mathcal{A}$.
\emph{This proves the claim}, showing that 
$\R^3\setminus \mathcal{A}$ has 
at most two components. To see that it has at least two, it is enough to know 
that any smooth cycle $\varphi:\mathbb{S}^1\to\R^3$ 
has intersection number $I_2(\varphi,\mathcal{A})=0$ modulo 2.
Indeed, if this is the case and if $x_0^+$ and $x_0^-$ could be joined
by a continuous arc $\gamma:[0,1]\to\R^3$ not intersecting $\mathcal{A}$, 
then $\gamma$ could be 
chosen $C^\infty$-smooth (see \cite[Ch.1, Sec. 6, Ex. 3]{GP})
and we could complete it
into a cycle $\varphi:\mathbb{S}^1\to\R^3$ 
by concatenation with the segment $[x_0^-,x_0^+]$ 
which intersects $\mathcal{A}$ exactly
once (at $x_0$), in a transversal manner. Elementary modifications
at $x^-_0$ and $x^+_0$ will arrange things so
that  $\varphi$ becomes $C^\infty$-smooth, and this 
would contradict the fact that the number of intersection points with
$\mathcal{A}$ must be even.
Now, if $\mathbb{D}$ is the unit disk,
any smooth map $\varphi:\mathbb{S}^1\to\R^3$ extends to a smooth map
$f:\overline{\mathbb{D}}\to\R^3$ (take for example
$f(re^{i\theta})=e^{1-1/r}\varphi(e^{i\theta})$). Thus, by 
the boundary theorem \cite[p. 80]{GP}, the intersection number modulo 2
of $\varphi$ 
with any smooth  and complete embedded submanifold of dimension 2 in $\R^3$ 
(in particular with $\mathcal{A}$ ) must be zero.
This achieves the proof.
\end{proof}

\end{document}